\documentclass[11pt,a4paper]{article}

\usepackage[margin=1in]{geometry}
\usepackage{amssymb,mathtools,amsthm}
\usepackage{algorithm,algorithmic}
\usepackage{array}
\usepackage{booktabs}
\usepackage{enumitem}
\usepackage{bm}
\usepackage{graphicx}
\usepackage{float}
\usepackage[section]{placeins}
\usepackage{microtype}
\usepackage{needspace}
\usepackage{titlesec}
\usepackage[hidelinks]{hyperref}

\allowdisplaybreaks
\titleformat{\section}{\large\bfseries}{\thesection}{0.7em}{}
\titleformat{\subsection}{\normalsize\bfseries}{\thesubsection}{0.7em}{}
\graphicspath{{./}}

\newtheorem{theorem}{Theorem}[section]
\newtheorem{lemma}[theorem]{Lemma}
\newtheorem{proposition}[theorem]{Proposition}
\newtheorem{corollary}[theorem]{Corollary}
\newtheorem{assumption}[theorem]{Assumption}

\theoremstyle{definition}

\theoremstyle{remark}
\newtheorem{remark}[theorem]{Remark}

\numberwithin{equation}{section}
\numberwithin{algorithm}{section}
\newcommand{\email}[1]{\href{mailto:#1}{\texttt{#1}}}
\newenvironment{keywords}{\par\smallskip\noindent\textbf{Keywords.}\ }{\par}
\newenvironment{MSCcodes}{\par\smallskip\noindent\textbf{Mathematics Subject Classification (2020).}\ }{\par}
\date{}

\title{Halpern-Accelerated Majorized ADMM with Optimal Non-Ergodic Convergence under Degenerate Preconditioning}

\author{Ning Zhang\thanks{School of Computer Science and Technology, Dongguan University of Technology, Dongguan 523808, China
(\email{zhangning@dgut.edu.cn}).}
\and Benqi Liu\thanks{Corresponding author.  Beijing International Center for Mathematical Research, Peking University, Beijing 100871, China
(\email{bqliu@pku.edu.cn}).}
\and Liwei Zhang\thanks{National Frontiers Science Center for Industrial Intelligence and Systems Optimization, Northeastern University, Shenyang
110819, China (\email{zhanglw@mail.neu.edu.cn}).}}

\hypersetup{
pdftitle={Halpern-Accelerated Majorized ADMM with Optimal Non-Ergodic Convergence under Degenerate Preconditioning},
pdfauthor={Ning Zhang, Benqi Liu, and Liwei Zhang}
}

\begin{document}
\maketitle

\begin{abstract}
We develop a Halpern-accelerated majorized alternating direction method of multipliers (ADMM) with possibly indefinite proximal terms for linearly
constrained convex composite optimization. Majorization simplifies the subproblems but introduces a forward perturbation to the standard degenerate proximal point representation.
We reformulate majorized ADMM as a perturbed degenerate proximal point method and characterize the fixed points of the induced mapping as
Karush--Kuhn--Tucker (KKT) solutions. Under range compatibility and metric cocoercivity conditions, the mapping is nonexpansive in the seminorm induced by a possibly singular preconditioner. In finite dimensions, a verifiable block-matrix condition permits indefinite proximal regularization while retaining the required metric properties. A relaxed Halpern iteration then yields an $\mathcal O(1/k)$ fixed-point residual bound, sharp in the general nonexpansive setting, and an $\mathcal O(1/k)$ nonergodic KKT residual bound at the intermediate iterates. Numerical experiments illustrate the predicted residual decay and the benefits of majorization and indefinite proximal terms.
\end{abstract}

\begin{keywords}
majorized alternating direction method of multipliers, Halpern iteration, degenerate preconditioning, indefinite proximal term, nonergodic convergence, Karush--Kuhn--Tucker residual
\end{keywords}

\begin{MSCcodes}
90C25, 65K05, 47H05, 49M27
\end{MSCcodes}

\section{Introduction}

We develop a Halpern-accelerated majorized alternating direction method of multipliers (ADMM) for the following linearly constrained convex composite
optimization problem:
\begin{equation}\label{model}
\min_{y\in\mathbb{Y},\,z\in\mathbb{Z}}
p(y)+f(y)+q(z)+g(z)
\quad
\mathrm{s.t.}
\quad
Ay+Bz=c,
\end{equation}
where $\mathbb{X}$, $\mathbb{Y}$, and $\mathbb{Z}$ are finite-dimensional Euclidean spaces, each equipped with an inner product
$\langle\cdot,\cdot\rangle$ and its induced norm $\|\cdot\|$. The functions $p$ and $q$ are closed proper convex functions, and $f$ and $g$
are convex functions with Lipschitz continuous gradients. The operators $A:\mathbb Y\to\mathbb X$ and $B:\mathbb Z\to\mathbb X$
are linear, and $c\in\mathbb X$ is given.

ADMM and its variants have been extensively studied and widely applied to structured separable convex optimization problems arising in signal processing, image science, machine learning, and operations research; see, e.g., \cite{Chen2012,chen2019unified,fazel2013hankel,fortin2000augmented,gabay1976dual,glowinski2013numerical,gu2014customized,han2018linear,li2016majorized,xiao2018generalized,zhang2020linearly} and the references therein. It is well known that the classical ADMM with unit step size is equivalent to the Douglas--Rachford splitting method applied to the dual inclusion problem, and the Douglas--Rachford splitting can in turn be interpreted as a special instance of the proximal point algorithm. This operator-theoretic perspective has played a fundamental role in the convergence analysis of the classical ADMM and has inspired the development of many modern operator-splitting algorithms.

Beyond this classical setting, a variety of enhanced ADMM-type methods have been developed by incorporating larger dual step sizes together with positive semidefinite or even indefinite proximal terms. For these generalized schemes, however, the classical operator-theoretic interpretation is generally no longer available. Consequently, the convergence analysis has been carried out primarily within the framework of convex optimization by exploiting the structure of the underlying primal-dual formulation, rather than through abstract operator-splitting techniques. Under standard convexity assumptions, global convergence has been well established, while linear convergence can be obtained under suitable error bound conditions, which are known to hold automatically for many important applications such as Lasso, fused Lasso, and group Lasso. In addition, the ergodic iteration complexity of order $O(1/k)$ has also been extensively studied.

In recent years, acceleration schemes based on the Halpern iteration have attracted growing attention. Halpern's fixed-point
iteration \cite{Hal1967} was originally proposed for approximating fixed points of nonexpansive mappings, and its sharp $\mathcal{O}(1/k)$ fixed-point residual bound in Hilbert spaces was established by Lieder \cite{lieder2021convergence}. Building upon Lieder's convergence analysis, Zhang \emph{et al.} \cite{zhang2025halpern} extended the Halpern framework to
the inexact proximal point method by introducing practical inexactness criteria while preserving the $\mathcal{O}(1/k)$ fixed-point residual rate. On the other hand, Kim \cite{kim2021} proposed an accelerated proximal point method for maximally monotone operators, and Contreras and Cominetti \cite{contreras2023optimal} subsequently revealed that Kim's method admits an
equivalent interpretation as a Halpern iteration applied to the proximal point framework, thereby providing an operator-theoretic explanation for its acceleration mechanism.

The above developments have recently inspired the acceleration of a broad class of first-order methods. In particular, Halpern-type acceleration has been incorporated into preconditioned ADMM~\cite{sun2025accelerating} and the primal--dual hybrid gradient method~\cite{lu2024restarted}. It has also led to GPU-oriented solvers for large-scale linear programming, including HPR-LP and cuPDLPx~\cite{chen2025hpr,lu2025cupdlpx}. The key ingredient underlying these developments is that the corresponding first-order algorithms admit equivalent fixed-point formulations associated with nonexpansive operators, which makes Halpern-type acceleration directly applicable. A notable example is the degenerate preconditioned proximal point framework established by Bredies \emph{et al.} \cite{bredies2022degenerate}. Based on this framework, Sun \emph{et al.} \cite{sun2025accelerating} and Chen \emph{et al.} \cite{chen2025hpr} showed that ADMM with positive semidefinite proximal terms admits an equivalent fixed-point characterization through a degenerate proximal point mapping. This characterization enables Halpern-type acceleration and leads to optimal nonergodic convergence guarantees for the corresponding ADMM-type methods.

Despite this progress, many applications involve smooth, nonquadratic objective functions. Solving the resulting ADMM subproblems exactly can still be expensive. To address this issue, the majorization technique replaces the smooth functions $f$ and $g$ with quadratic surrogate models. Specifically, for any $y,y'\in\mathbb{Y}$ and $z,z'\in\mathbb{Z}$, one defines
\begin{align}
\hat{f}(y,y')
&:= f(y')+\langle\nabla f(y'),\,y-y'\rangle +\frac{1}{2}\|y-y'\|_{\widehat{\Sigma}_f}^{2},\label{eq:maj_f}\\
\hat{g}(z,z')
&:= g(z')+\langle\nabla g(z'),\,z-z'\rangle +\frac{1}{2}\|z-z'\|_{\widehat{\Sigma}_g}^{2},\label{eq:maj_g}
\end{align}
where $\widehat{\Sigma}_{f}$ and $\widehat{\Sigma}_{g}$ are self-adjoint positive semidefinite operators chosen to satisfy the majorization conditions, whose existence is guaranteed by the Lipschitz continuity of $\nabla f$ and $\nabla g$, respectively. This construction gives rise to the majorized ADMM framework, which substantially reduces the computational cost per iteration while preserving the favorable convergence properties of ADMM, especially when either $f$ or $g$ is nonlinear or nonquadratic \cite{chen2019unified,li2016majorized,zhang2020linearly}. In addition, Cui \emph{et al.} \cite{cui2016convergence} established a best-iterate nonergodic $\mathcal{O}(1/\sqrt{k})$ iteration-complexity bound for a majorized ADMM framework, which also covers the classical ADMM as a special case, in terms of an appropriate KKT residual measure. However, unlike the semidefinite proximal ADMM discussed above, the introduction of majorization generally destroys the exact fixed-point characterization arising from the degenerate proximal point framework. As a result, the existing operator-theoretic acceleration techniques are no longer directly applicable.

Throughout the paper, ``majorization'' means that
\begin{equation}\label{eq:majorization-inequalities}
f(y)\leq \hat f(y,y'),\qquad g(z)\leq \hat g(z,z') \quad \text{for all }y,y'\in\mathbb Y\text{ and }z,z'\in\mathbb Z.
\end{equation}

The key challenge is to recover an operator-theoretic framework for the majorized ADMM after the introduction of majorization breaks the existing equivalence with the degenerate proximal point method. Our main technical contribution is to show that the majorized ADMM can still be reformulated as a perturbed degenerate proximal point method, thereby recovering a fixed-point characterization that enables the subsequent operator-theoretic convergence analysis and Halpern acceleration.

In this paper, we bridge this gap by establishing a new operator-theoretic characterization for majorized ADMM, through which the majorized ADMM iteration is reformulated as a perturbed degenerate proximal point method. This characterization enables us to extend the Halpern acceleration framework from semidefinite proximal ADMM to the majorized setting. Building upon this framework, we establish the nonexpansiveness of the associated fixed-point mapping, prove global convergence of the standard fixed-point iteration in finite dimensions, and derive an optimal $\mathcal{O}(1/k)$ nonergodic fixed-point residual bound together with an $\mathcal{O}(1/k)$ KKT residual bound at the intermediate iterates. To the best of our knowledge, this is the first work that systematically develops an operator-theoretic framework for Halpern acceleration of majorized ADMM.

The main contributions of this paper are threefold. First, we establish a novel operator-theoretic characterization for the majorized ADMM by reformulating its iteration as a perturbed degenerate proximal point method. This characterization extends the existing degenerate proximal point framework from semidefinite proximal ADMM to the majorized setting, thereby bridging the gap between majorized ADMM and fixed-point methods. Second, based on the proposed operator-theoretic framework, we establish the nonexpansiveness of the associated fixed-point mapping in the degenerate metric induced by the preconditioner under the explicit range-compatibility
condition stated in Assumption~\ref{assume-coco}. This yields global convergence of the standard iteration without separately assuming ambient boundedness or solution-valued cluster points. Finally, we develop a Halpern-accelerated majorized ADMM and establish the optimal $\mathcal{O}(1/k)$ convergence rate for its fixed-point
residual. We further show that the associated KKT residual admits an $\mathcal{O}(1/k)$ upper bound at the intermediate iterates.

The remainder of this paper is organized as follows. In Section~2, we establish a novel operator-theoretic characterization of the majorized ADMM by reformulating its iteration as a perturbed degenerate proximal point method. Section~3 develops the operator-theoretic convergence analysis of the associated fixed-point iteration by establishing the
nonexpansiveness of the induced fixed-point mapping and its global convergence under a degenerate metric. In Section~4, we propose a Halpern-accelerated majorized ADMM and establish the optimal $\mathcal{O}(1/k)$ nonergodic fixed-point residual rate. Section~5 presents numerical experiments that illustrate the theoretical
results and compare different algorithmic mechanisms under controlled settings.
Section~\ref{sec:conclusion} concludes the paper.

\noindent
{\bf Notations.}  For a linear operator $\Sigma$, we denote by $\operatorname{Im}(\Sigma)$ and $\operatorname{ker}(\Sigma)$ its range and null space, respectively. The symbol $\Sigma^\dagger$ denotes the Moore--Penrose pseudoinverse of $\Sigma$. For a subspace $\mathcal U$, $\mathcal U^\perp$ denotes its orthogonal complement.  Let $\mathbb H:=\mathbb Y\times\mathbb Z\times\mathbb X$, and denote the primal-dual variable associated with problem \eqref{model} by $w=(y,z,x)\in\mathbb H$.

\section{An Operator-Theoretic Reformulation of the Majorized ADMM}\label{sec2:operator-theoretic}

The majorization technique destroys the operator-theoretic characterization that underlies the existing Halpern-accelerated ADMM methods. The objective of this section is to recover such a characterization by reformulating the majorized ADMM as a perturbed degenerate proximal point method. This reformulation extends the existing operator-theoretic framework developed for semidefinite proximal ADMM to the majorized setting and serves as the foundation for the convergence analysis and acceleration results established in the subsequent sections.

Following the iteration framework studied in \cite{sun2025accelerating}, we adopt the updating order
$z \rightarrow x \rightarrow y$, which coincides with the generalized ADMM proposed in \cite{Chen2012,xiao2018generalized}.
The intermediate mapping of the majorized ADMM (MajADMM) defined below uses
the unit multiplier update.
The separate outer relaxation parameter $\rho$ in Algorithm~\ref{alg:halpern} is introduced after the fixed-point reformulation in Section~\ref{sec:HMADMM}. Specifically, given a penalty parameter $\sigma>0$, proximal operators $\mathcal S$ and $\mathcal T$, and an initial point $w^0=(y^0,z^0,x^0)$, the iteration is given by
\begin{subequations}\label{alg-admm}
\begin{align}
z^{k+1}
&= \arg\min_{z} \Bigl\{q(z)+\hat{g}(z,z^k)
  +\langle B^*x^k,z\rangle 
  +\frac{\sigma}{2}\|Ay^k+Bz-c\|^2
  +\frac{1}{2}\|z-z^k\|_{\mathcal T}^2\Bigr\}, \label{alg-z}\\[2mm]
x^{k+1}
&=x^k+\sigma(Ay^k+Bz^{k+1}-c), \label{alg-x}\\[2mm]
y^{k+1}
&= \arg\min_{y} \Bigl\{p(y)+\hat{f}(y,y^k)
  +\langle A^*x^{k+1},y\rangle 
  +\frac{\sigma}{2}\|Ay+Bz^{k+1}-c\|^2
  +\frac{1}{2}\|y-y^k\|_{\mathcal S}^2\Bigr\}, \label{alg-y}
\end{align}
\end{subequations}
where the functions $\hat{f}$ and $\hat{g}$ are defined in \eqref{eq:maj_f} and \eqref{eq:maj_g}, respectively.
The key step in our analysis is to show that the above iteration admits an equivalent operator-theoretic reformulation. To this end, we first introduce the optimality system associated with problem \eqref{model}.

The Karush-Kuhn-Tucker (KKT) optimality condition associated with problem \eqref{model} can be written as  \(0 \in {\cal R}(w)\), where the exact maximal monotone operator \(\cal R\) is defined as:
\begin{equation}\label{def-R}
{\cal R}(w) = \begin{pmatrix}
\partial p(y) + \nabla f(y) + A^* x \\[2pt]
\partial q(z) + \nabla g(z) + B^* x \\[2pt]
- Ay - Bz + c
\end{pmatrix}.
\end{equation}To facilitate the subsequent operator-theoretic analysis, we decompose the KKT operator into the following two components:
\begin{equation}\label{def-R-F}
\mathcal{R}_{ns}(w)=
\begin{pmatrix}
\partial p(y)+A^*x\\[2pt]
\partial q(z)+B^*x\\[2pt]
-Ay-Bz+c
\end{pmatrix} \,\,\hbox{and}\,\, \mathcal{F}_{sm}(w)=
\begin{pmatrix}
\nabla f(y)\\[2pt]
\nabla g(z)\\[2pt]
0
\end{pmatrix}.
\end{equation}Accordingly,
\begin{equation}\label{17}
\mathcal R
=
\mathcal R_{ns}
+
\mathcal F_{sm},
\end{equation}and the KKT condition is equivalently written as
$$
0\in
\mathcal R_{ns}(w^*)
+
\mathcal F_{sm}(w^*).
$$
We write the KKT solution set as
\begin{equation}\label{def-Omega}
\Omega^*:=\operatorname{zer}(\mathcal R)
=\{w\in\mathbb H:0\in\mathcal R_{ns}(w)+\mathcal F_{sm}(w)\}.
\end{equation}
This decomposition separates the maximal monotone and smooth components of the KKT operator, which is essential for constructing the perturbed degenerate proximal point reformulation developed below. Following \cite{sun2025accelerating}, we introduce the symmetric preconditioning operator
\begin{equation}\label{def-M}
\mathcal{M} = \begin{pmatrix}
\mathcal{S} + \widehat{\Sigma}_f + \sigma A^* A & 0 & A^* \\[2pt]
0 & \mathcal{T} + \widehat{\Sigma}_g & 0 \\[2pt]
A & 0 & {\sigma}^{-1} I
\end{pmatrix}.
\end{equation}
For later use, set $\mathcal P_f:=\mathcal S+\widehat\Sigma_f$ and $\mathcal P_g:=\mathcal T+\widehat\Sigma_g$.

\begin{assumption}\label{assume-wellposed}
The KKT solution set $\Omega^*$ is nonempty, and
\begin{equation}\label{eq:wellposed-curvature}
\mathcal P_f\succeq0,\,\, \mathcal P_g\succeq0,\,\, \mathcal P_f+\sigma A^*A\succ0,\,\, \text{and}\,\,
\mathcal P_g+\sigma B^*B\succ0.
\end{equation}
\end{assumption}

Assumption~\ref{assume-wellposed} permits the proximal operators $\mathcal S$ and $\mathcal T$ themselves to be indefinite. It requires only that the operators formed by combining the majorization and proximal
terms be positive semidefinite and that the two subproblem Hessian operators in \eqref{eq:wellposed-curvature} be positive definite. Consequently, both primal subproblems have unique solutions and Proposition~\ref{prop:DPPA-Equiv} gives $\mathcal M\succeq0$.

We are now ready to establish the desired operator-theoretic reformulation of the majorized ADMM. The following proposition shows that the iteration \eqref{alg-admm} can be equivalently represented as a perturbed degenerate proximal point method.

\begin{proposition}\label{prop:DPPA-Equiv}
Let $\{w^k\}$ be the sequence generated by \eqref{alg-admm}. Then it satisfies the perturbed degenerate proximal point inclusion
\begin{equation}\label{dPPA}
0\in \mathcal R(w^{k+1})+\mathcal M(w^{k+1}-w^k)+{\cal E}^k,
\end{equation}
where ${\mathcal E}^k:=\mathcal F_{sm}(w^k)-\mathcal F_{sm}(w^{k+1})$.
Equivalently,
\begin{equation}\label{dPPA-equivalent}
0\in \mathcal R_{ns}(w^{k+1})+\mathcal F_{sm}(w^k)
+\mathcal M(w^{k+1}-w^k).
\end{equation}
Moreover, if $\mathcal S+\widehat{\Sigma}_f\succeq0$ and
$\mathcal T+\widehat{\Sigma}_g\succeq0$, then $\mathcal M\succeq0$.
\end{proposition}

\begin{proof} 
The proof follows by examining the first-order optimality conditions corresponding to each update \eqref{alg-admm}. We consider the updates one by one.

For the $z$-update, the first-order optimality condition of subproblem \eqref{alg-z} is
$$
0 \in \partial q(z^{k+1})
+\nabla g(z^k)
+\widehat{\Sigma}_g(z^{k+1}-z^k)
+B^*x^k
+\sigma B^*(Ay^k+Bz^{k+1}-c)
+\mathcal{T}(z^{k+1}-z^k).
$$
By the $x$-update \eqref{alg-x}, the above inclusion can be equivalently written as
$$
0 \in \partial q(z^{k+1})
+\nabla g(z^k)
+B^*x^{k+1}
+(\mathcal{T}+\widehat{\Sigma}_g)(z^{k+1}-z^k).
$$
It follows that
\begin{equation}\label{th2.1-eq-10}
\begin{array}{l}
0 \in \underbrace{\partial q(z^{k+1}) +\nabla g(z^{k+1})
+B^*x^{k+1}}_{\mathcal{R}_z(w^{k+1})}\\
~~~~~+\underbrace{(\mathcal{T}+\widehat{\Sigma}_g)(z^{k+1}-z^k)}_{\mathcal{M}_{zz}}
+\underbrace{\nabla g(z^k)-\nabla g(z^{k+1})}_{\mathcal{E}_z^k}.
\end{array}
\end{equation}

Second, for the $x$-update, rearranging \eqref{alg-x} yields
$$
\sigma^{-1}(x^{k+1}-x^k)
=Ay^k+Bz^{k+1}-c
=(Ay^{k+1}+Bz^{k+1}-c)\mathbin{+}{A(y^k-y^{k+1})},
$$
and hence
\begin{equation}\label{th2.1-eq-11}
0=
\underbrace{-Ay^{k+1}-Bz^{k+1}+c}_{\mathcal{R}_x(w^{k+1})}
+\underbrace{A}_{\mathcal{M}_{xy}}(y^{k+1}-y^k)
+\underbrace{\sigma^{-1}I}_{\mathcal{M}_{xx}}(x^{k+1}-x^k).
\end{equation}

Third, for the $y$-update, the first-order optimality condition of subproblem \eqref{alg-y} is
$$
\begin{aligned}
0 \in\;&
\partial p(y^{k+1})
+\nabla f(y^k)
+\widehat{\Sigma}_f (y^{k+1}-y^k)
+A^*x^{k+1} \\
&+\sigma A^*(Ay^{k+1}+Bz^{k+1}-c)
+\mathcal{S}(y^{k+1}-y^k).
\end{aligned}
$$
Using \eqref{alg-x}, the above inclusion can be rewritten as
$$
\begin{aligned}
0 \in\;&
\partial p(y^{k+1})
+\nabla f(y^k)
+A^*x^{k+1}
+A^*(x^{k+1}-x^k)\\
&+\sigma A^*A(y^{k+1}-y^k)
+(\mathcal{S}+\widehat{\Sigma}_f)(y^{k+1}-y^k).
\end{aligned}
$$
It follows that
\begin{equation}\label{th2.1-eq-12}
\begin{aligned}
0 \in\;&
\underbrace{\partial p(y^{k+1})
+\nabla f(y^{k+1})
+A^*x^{k+1}}_{\mathcal{R}_y(w^{k+1})}
+\underbrace{(\mathcal{S}+\widehat{\Sigma}_f+\sigma A^*A)
(y^{k+1}-y^k)}_{\mathcal{M}_{yy}} \\
&+
\underbrace{A^*}_{\mathcal{M}_{yx}}
(x^{k+1}-x^k)
+\underbrace{\nabla f(y^k)-\nabla f(y^{k+1})}_{\mathcal{E}_y^k}.
\end{aligned}
\end{equation}
Combining \eqref{th2.1-eq-10}--\eqref{th2.1-eq-12} yields exactly the matrix-vector representation \eqref{dPPA}.

Finally, it remains to show that $\mathcal{M}$ is positive semidefinite. For any
$w=(y,z,x)$, we have
$$
\begin{aligned}
\langle w,\mathcal{M}w\rangle
&=
\langle y,(\mathcal{S}+\widehat{\Sigma}_f+\sigma A^*A)y\rangle
+\langle z,(\mathcal{T}+\widehat{\Sigma}_g)z\rangle \\
&\quad
+\langle x,\sigma^{-1}x\rangle
+\langle y,A^*x\rangle
+\langle x,Ay\rangle \\
&=
\|y\|_{\mathcal{S}+\widehat{\Sigma}_f}^2
+\|z\|_{\mathcal{T}+\widehat{\Sigma}_g}^2
+\sigma\|Ay\|^2
+\sigma^{-1}\|x\|^2
+2\langle Ay,x\rangle \\
&=
\|y\|_{\mathcal{S}+\widehat{\Sigma}_f}^2
+\|z\|_{\mathcal{T}+\widehat{\Sigma}_g}^2
+\|\sigma^{1/2}Ay+\sigma^{-1/2}x\|^2.
\end{aligned}
$$
Since
$\mathcal{S}+\widehat{\Sigma}_f \succeq 0$
and
$\mathcal{T}+\widehat{\Sigma}_g \succeq 0$,
it follows immediately that
$\langle w,\mathcal{M}w\rangle \ge 0$. Hence, $\mathcal{M}\succeq 0$. The proof is complete.
\end{proof}

The reformulation in Proposition~\ref{prop:DPPA-Equiv} suggests defining a fixed-point mapping through the inverse of $\mathcal M+\mathcal R_{ns}$. For this construction to be well posed, we need the corresponding resolvent-type mapping to be single-valued. The following proposition establishes this property, together with global Lipschitz continuity, under a natural positive definiteness condition.

\begin{proposition}\label{prop-resolvent}
Under Assumption~\ref{assume-wellposed}, the inverse $(\mathcal M+\mathcal R_{ns})^{-1}:\mathbb H\to\mathbb H$ has full domain, is single-valued, and is globally Lipschitz continuous.
\end{proposition}

\begin{proof}
Set
$Q_y:=\mathcal P_f+\sigma A^*A\,\,\text{and}\,\, Q_z:=\mathcal P_g+\sigma B^*B.$
For $v=(v_y,v_z,v_x)\in\mathbb H$, the inclusion $v\in(\mathcal M+\mathcal R_{ns})(y,z,x)$ is equivalent to
\begin{subequations}\label{eq:resolvent-blocks}
\begin{align}
v_y&\in\partial p(y)+Q_y y+2A^*x,\label{eq:resolvent-y}\\
v_z&\in\partial q(z)+\mathcal P_g z+B^*x,\label{eq:resolvent-z}\\
v_x&=-Bz+c+\sigma^{-1}x.\label{eq:resolvent-x}
\end{align}
\end{subequations}
In particular, the $Ay$ terms in the third block of $\mathcal M$ and $\mathcal R_{ns}$ cancel.  Eliminating the variables in the order $z\to x\to y$ gives
\begin{subequations}\label{eq:resolvent-triangular}
\begin{align}
z&=(\partial q+Q_z)^{-1}
\bigl(v_z-\sigma B^*(v_x-c)\bigr),\label{eq:resolvent-triangular-z}\\
x&=\sigma(v_x-c+Bz),\label{eq:resolvent-triangular-x}\\
y&=(\partial p+Q_y)^{-1}(v_y-2A^*x).
\label{eq:resolvent-triangular-y}
\end{align}
\end{subequations}

By Assumption~\ref{assume-wellposed}, $Q_y\succ0$ and $Q_z\succ0$. Hence $\partial p+Q_y$ and $\partial q+Q_z$ are strongly monotone and surjective, and their inverses are everywhere defined, single-valued, and globally Lipschitz. Thus \eqref{eq:resolvent-triangular} defines exactly one $(y,z,x)$ for every $v\in\mathbb H$, which proves that the mapping is well-defined and single-valued on the entire space.

For completeness, let $\mu_y:=\lambda_{\min}(Q_y)>0$ and
$\mu_z:=\lambda_{\min}(Q_z)>0$.  For two inputs $v_1,v_2$ and their images
$w_i=(y_i,z_i,x_i)$, strong monotonicity and
\eqref{eq:resolvent-triangular} yield
\begin{align*}
&\|z_1-z_2\| \leq \mu_z^{-1}
\bigl(\|v_{1z}-v_{2z}\|+\sigma\|B\|\|v_{1x}-v_{2x}\|\bigr),\\
&\|x_1-x_2\| \leq \sigma\|v_{1x}-v_{2x}\|+\sigma\|B\|\|z_1-z_2\|,\\
&\|y_1-y_2\| \leq \mu_y^{-1}\bigl(\|v_{1y}-v_{2y}\|+2\|A\|\|x_1-x_2\|\bigr).
\end{align*}
Combining these three estimates gives a constant $L_J>0$ such that
$$
\|(\mathcal M+\mathcal R_{ns})^{-1}v_1
 -(\mathcal M+\mathcal R_{ns})^{-1}v_2\|
\leq L_J\|v_1-v_2\|.
$$
This completes the proof.
\end{proof}

With Proposition~\ref{prop-resolvent}, the reformulation in Proposition~\ref{prop:DPPA-Equiv} can now be expressed as a fixed-point iteration. Indeed, from \eqref{dPPA-equivalent},
we obtain
$$
(\mathcal M-\mathcal F_{sm})w^k \in
(\mathcal M+\mathcal R_{ns})w^{k+1}.
$$
This motivates the definition of the composite mapping
\begin{equation}\label{def-T}
\mathbf T
:=
(\mathcal M+\mathcal R_{ns})^{-1}(\mathcal M-\mathcal F_{sm}).
\end{equation}
Thus, the perturbed degenerate proximal point reformulation equivalently
yields a preconditioned forward--backward mapping, with $\mathcal F_{sm}$
evaluated explicitly and $\mathcal R_{ns}$ treated implicitly.

Under the assumptions of Proposition~\ref{prop-resolvent}, the mapping $\mathbf T$ is well defined. Consequently, the operator-theoretic reformulation developed above naturally induces a fixed-point mapping associated with the majorized ADMM iteration framework \eqref{alg-admm}. Define
$$
\operatorname{Fix}(\mathbf T):=\{\,w:\mathbf T(w)=w\,\}.
$$

The following theorem shows that this fixed-point mapping provides an exact characterization of the KKT solution set of problem~\eqref{model}. This result serves as the cornerstone for the operator-theoretic convergence analysis developed in the next section.

\begin{theorem}\label{th-fixM}
Under Assumption~\ref{assume-wellposed}, the fixed-point set of $\mathbf T$ coincides with the KKT solution set, that is, 
$\operatorname{Fix}(\mathbf T)=\Omega^*$.
\end{theorem}

\begin{proof}
By Proposition~\ref{prop-resolvent}, $\mathbf T$ is a single-valued mapping on
$\mathbb H$.  For any $w\in\mathbb H$,
\begin{align*}
w=\mathbf T w &\quad\Longleftrightarrow\quad (\mathcal M-\mathcal F_{sm})w \in(\mathcal M+\mathcal R_{ns})w\\
&\quad\Longleftrightarrow\quad 0\in\mathcal R_{ns}(w)+\mathcal F_{sm}(w)\\
&\quad\Longleftrightarrow\quad w\in\Omega^*.
\end{align*}
The cancellation of $\mathcal Mw$ is valid without requiring $\mathcal M$ to be invertible.
\end{proof}

Theorem~\ref{th-fixM} completes the fixed-point characterization of the majorized ADMM by establishing the exact correspondence between the fixed points of $\mathbf T$ and the KKT solutions of problem~\eqref{model}. This characterization provides the foundation for the subsequent analysis of the mapping $\mathbf T$ and the convergence properties of the associated fixed-point iteration.

\section{Convergence Analysis of the Fixed-Point Iteration}

Building upon the operator-theoretic reformulation established in Section~2, we now study the convergence properties of the sequence $\{w^k\}$ generated by the fixed-point iteration
$$
w^{k+1}=\mathbf T(w^k), \qquad k=0,1,\ldots,
$$
starting from an arbitrary initial point $w^0\in\mathbb H$. The main challenge is that the preconditioning operator $\mathcal M$ is allowed to be positive semidefinite and hence may induce only a seminorm, so standard fixed-point arguments in a Hilbertian norm cannot be applied directly. To address this degeneracy, we first introduce a factorization of $\mathcal M$ and formulate an appropriate cocoercivity condition for the smooth operator $\mathcal F_{sm}$ with respect to the Moore--Penrose pseudoinverse of $\mathcal M$. These ingredients will then be used to establish the nonexpansiveness of $\mathbf T$ in the $\mathcal M$-seminorm and the convergence of the resulting fixed-point iteration.

\begin{proposition}\cite[Proposition 2.3]{bredies2022degenerate}
Let \(\mathcal{M} : \mathbb{H} \to \mathbb{H}\) be a linear, bounded, self-adjoint, and positive semidefinite operator. 
Then, there exists a bounded and injective operator \({\cal C} : \mathbb{U} \to \mathbb{H}\) for some real Hilbert space \(\mathbb{U}\), such that 
$$
\mathcal{M} = {\cal C}{\cal C}^*,
$$
where \({\cal C}^* : \mathbb{H} \to \mathbb{U}\) is the adjoint of \({\cal C}\). 
Moreover, if \(\mathcal{M}\) has closed range, then \({\cal C}^*\) is onto.
\end{proposition}

\begin{assumption}\label{assume-coco}
There exists $\beta\in(0,1]$ such that, for every $w_1,w_2\in\mathbb H$,
the variations
$$
\Delta w:=w_1-w_2\,\,\text{and}\,\, \Delta\mathcal F_{sm}:=\mathcal F_{sm}(w_1)-\mathcal F_{sm}(w_2)
$$
satisfy both the range-compatibility condition
\begin{equation}\label{assume-range}
\Delta\mathcal F_{sm}\in\operatorname{Im}(\mathcal M)
\end{equation}
and the metric cocoercivity inequality
\begin{equation}\label{assume-coco-Ieq}
\langle \Delta\mathcal F_{sm},\Delta w\rangle
\geq
\frac{1}{2\beta}
\|\Delta\mathcal F_{sm}\|_{\mathcal M^\dagger}^{2}.
\end{equation}
Here $\mathcal M^\dagger$ is the Moore--Penrose pseudoinverse of $\mathcal M$.
\end{assumption}

\begin{remark}\label{rem:coco}
Finite dimensionality makes $\operatorname{Im}(\mathcal M)$ closed.
Condition \eqref{assume-range} is nevertheless an
essential part of Assumption~\ref{assume-coco}: it makes the generalized
Cauchy--Schwarz and Young inequalities applicable to every smooth-operator
variation.  The required condition is on
differences of $\mathcal F_{sm}$, rather than on the absolute range
$\operatorname{Im}(\mathcal F_{sm})$.
\end{remark}

\begin{proposition}\label{prop:matrix-sufficient-coco}
Let $\widehat\Sigma_f$ and $\widehat\Sigma_g$ be the self-adjoint positive semidefinite majorization operators in \eqref{eq:maj_f}--\eqref{eq:maj_g}. Assume that there exists $\beta\in(0,1]$ such that
\begin{equation}\label{eq:matrix-sufficient-cond}
\mathcal P_f\succeq \frac{1}{2\beta}\widehat\Sigma_f\,\,\,\text{and}\,\, \,\mathcal P_g\succeq \frac{1}{2\beta}\widehat\Sigma_g.
\end{equation}
Then $\mathcal P_f\succeq0$, $\mathcal P_g\succeq0$, and both \eqref{assume-range} and \eqref{assume-coco-Ieq} hold.  Thus
Assumption~\ref{assume-coco} is satisfied.
\end{proposition}

\begin{proof}
First, we show that the gradient difference of $f$ lies in the range of $\widehat\Sigma_f$. Let
$$
\Delta y:=y_1-y_2\,\,\text{and}\,\, \Delta G_f:=\nabla f(y_1)-\nabla f(y_2).
$$
For any $d\in\ker(\widehat\Sigma_f)$, the majorization inequality \eqref{eq:majorization-inequalities} gives
$$
f(y+td)\leq f(y)+t\langle\nabla f(y),d\rangle,\qquad\forall y\in\mathbb Y,\quad \forall t\in\mathbb R.
$$
On the other hand, by convexity of the function $f$,
$$
f(y+td)\geq f(y)+t\langle\nabla f(y),d\rangle.
$$
Thus equality holds for every $t\in\mathbb R$, which implies that $f$ is affine along each direction in $\ker(\widehat\Sigma_f)$. Hence
$$
\langle\nabla f(y_1)-\nabla f(y_2),d\rangle=0, \qquad
\forall d\in\ker(\widehat\Sigma_f).
$$
Since the space is finite-dimensional and $\widehat\Sigma_f$ is self-adjoint,
$\operatorname{Im}(\widehat\Sigma_f)=\ker(\widehat\Sigma_f)^\perp$. Therefore,
\begin{equation}\label{eq:range-sigma-f}
\Delta G_f\in\operatorname{Im}(\widehat\Sigma_f).
\end{equation}
Similarly, for $\Delta z:=z_1-z_2$  and $\Delta G_g:=\nabla g(z_1)-\nabla g(z_2)$, one has
\begin{equation}\label{eq:range-sigma-g}
\Delta G_g\in\operatorname{Im}(\widehat\Sigma_g).
\end{equation}

Next, by \cite[Lemma 3.2]{zhang2020linearly}, for any $y_1,y_2,\delta\in\mathbb Y$,
\begin{equation}\label{eq:maj-lemma-f}
\left\langle \nabla f(y_1)-\nabla f(y_2), \delta-y_2 \right\rangle \geq -\frac14 \|y_1-\delta\|_{\widehat\Sigma_f}^2 .
\end{equation}
Taking $\delta=y_1-2\widehat\Sigma_f^\dagger\Delta G_f$ in \eqref{eq:maj-lemma-f}, we get
$$
\left\langle \Delta G_f,\Delta y-2\widehat\Sigma_f^\dagger\Delta G_f
\right\rangle\geq-\frac{1}{4}\left\|2\widehat\Sigma_f^\dagger\Delta G_f\right\|_{\widehat\Sigma_f}^2.
$$
By \eqref{eq:range-sigma-f}, it holds that
$$
\left\|\widehat\Sigma_f^\dagger\Delta G_f\right\|_{\widehat\Sigma_f}^2 =\|\Delta G_f\|_{\widehat\Sigma_f^\dagger}^2.
$$
Thus,
$$
\left\langle \Delta G_f,\Delta y \right\rangle-2\|\Delta G_f\|_{\widehat\Sigma_f^\dagger}^2 \geq-\|\Delta G_f\|_{\widehat\Sigma_f^\dagger}^2,
$$
and hence
\begin{equation}\label{eq:f-coco-Sigma}
\left\langle \Delta G_f,\Delta y \right\rangle \geq \|\Delta G_f\|_{\widehat\Sigma_f^\dagger}^2.
\end{equation}
Similarly, using
\eqref{eq:majorization-inequalities}, \eqref{eq:range-sigma-g}, and \cite[Lemma 3.2]{zhang2020linearly}, one has
\begin{equation}\label{eq:g-coco-Sigma}
\left\langle \Delta G_g,\Delta z\right\rangle\geq \|\Delta G_g\|_{\widehat\Sigma_g^\dagger}^2.
\end{equation}

Since $\mathcal P_f\succeq \frac{1}{2\beta}\widehat\Sigma_f\succeq0$, for any
$d\in\ker(\mathcal P_f)$, we have
$$
0 = \langle d,\mathcal P_fd\rangle \geq \frac{1}{2\beta} \langle d,\widehat\Sigma_f d\rangle
\geq 0.
$$
Hence $\langle d,\widehat\Sigma_f d\rangle=0$. Furthermore, $\widehat\Sigma_f\succeq0$ implies
$\widehat\Sigma_f d=0$, and therefore
$$
\ker(\mathcal P_f)\subseteq\ker(\widehat\Sigma_f).
$$
Taking orthogonal complements and using finite dimensionality and self-adjointness yields
$$
\operatorname{Im}(\widehat\Sigma_f)= \ker(\widehat\Sigma_f)^\perp\subseteq \ker(\mathcal P_f)^\perp=\operatorname{Im}(\mathcal P_f).
$$
This, together with \eqref{eq:range-sigma-f},  implies
$$
\Delta G_f\in\operatorname{Im}(\mathcal P_f).
$$
the same argument gives $\Delta G_g\in\operatorname{Im}(\mathcal P_g)$.

For any
$u\in\operatorname{Im}(\widehat\Sigma_f)\subseteq \operatorname{Im}(\mathcal P_f)$, the variational characterization of the
Moore--Penrose inverse \cite{benisrael2003generalized} implies
$$
\|u\|_{\mathcal P_f^\dagger}^2 =\sup_{\xi\in\mathbb Y} \left\{2\langle u,\xi\rangle-\|\xi\|_{\mathcal P_f}^2\right\}.
$$
Using the fact that $\|\xi\|_{{\cal P}_f}^2\geq\frac{1}{2\beta}\|\xi\|_{\widehat\Sigma_f}^2$, we obtain
\begin{equation}\label{eq:pseudo-order-f}
\|u\|_{{\cal P}_f^\dagger}^2\leq \sup_{\xi\in\mathbb Y}\left\{2\langle u,\xi\rangle-\frac{1}{2\beta}\|\xi\|_{\widehat\Sigma_f}^2\right\} 
=2\beta\|u\|_{\widehat\Sigma_f^\dagger}^2.
\end{equation}

Applying \eqref{eq:pseudo-order-f} to $u=\Delta G_f$ and using \eqref{eq:f-coco-Sigma}, we obtain
\begin{equation}\label{eq:f-coco-Pf}
\left\langle \Delta G_f,\Delta y \right\rangle \geq \frac{1}{2\beta} \|\Delta G_f\|_{{\cal P}_f^\dagger}^2.
\end{equation} 
Similarly,
\begin{equation}\label{eq:g-coco-Pg}
\left\langle \Delta G_g,\Delta z \right\rangle
\geq \frac{1}{2\beta} \|\Delta G_g\|_{{\cal P}_g^\dagger}^2.
\end{equation}

It remains to relate these estimates to the $\mathcal M^\dagger$-metric. Recall that
$$
\mathcal M=
\begin{pmatrix}
{\cal P}_f+\sigma A^*A & 0 & A^*\\
0 & {\cal P}_g & 0\\
A & 0 & \sigma^{-1}I
\end{pmatrix}.
$$
For any $s=(\xi,\eta,\zeta)$, direct calculation gives
\begin{equation}\label{eq:M-quadratic-form}
\|s\|_{\mathcal M}^2 = \|\xi\|_{{\cal P}_f}^2 + \|\eta\|_{{\cal P}_g}^2 + \|\sigma^{1/2}A\xi+\sigma^{-1/2}\zeta\|^2 .
\end{equation}
Let $u\in\operatorname{Im}({\cal P}_f)$ and $v\in\operatorname{Im}({\cal P}_g)$.  Then $(u,v,0)\in\operatorname{Im}(\mathcal M)$. Indeed, choose $\xi$ and $\eta$ such that
$$
{\cal P}_f\xi=u, \,\, {\cal P}_g\eta=v,
$$
and set $\zeta=-\sigma A\xi$. Then
$$
\mathcal M \begin{pmatrix}
\xi\\
\eta\\
\zeta
\end{pmatrix}
=
\begin{pmatrix}
u\\
v\\
0
\end{pmatrix}.
$$
Therefore, by the  characterization of the Moore--Penrose inverse and \eqref{eq:M-quadratic-form},
\begin{align}
\|(u,v,0)\|_{\mathcal M^\dagger}^2
&=\sup_{\xi,\eta,\zeta}\left\{2\langle u,\xi\rangle+2\langle v,\eta\rangle -\|(\xi,\eta,\zeta)\|_{\mathcal M}^2\right\} \notag\\
&=\sup_{\xi,\eta,\zeta}\Big\{2\langle u,\xi\rangle
+2\langle v,\eta\rangle -\|\xi\|_{{\cal P}_f}^2-\|\eta\|_{{\cal P}_g}^2 -\|\sigma^{1/2}A\xi+\sigma^{-1/2}\zeta\|^2\Big\}.
\label{eq:Mdagger-var}
\end{align}
For fixed $\xi$ and $\eta$, the last term in \eqref{eq:Mdagger-var} is maximized by choosing $\zeta=-\sigma A\xi$, which makes the nonnegative square vanish. Thus
\begin{align}
\|(u,v,0)\|_{\mathcal M^\dagger}^2&=  \sup_{\xi,\eta}\left\{2\langle u,\xi\rangle+2\langle v,\eta\rangle -\|\xi\|_{{\cal P}_f}^2-\|\eta\|_{{\cal P}_g}^2\right\} \notag\\
&=\sup_{\xi} \left\{2\langle u,\xi\rangle-\|\xi\|_{{\cal P}_f}^2\right\} +\sup_{\eta}\left\{2\langle v,\eta\rangle-\|\eta\|_{{\cal P}_g}^2\right\} \notag\\
&=\|u\|_{{\cal P}_f^\dagger}^2+\|v\|_{{\cal P}_g^\dagger}^2 .
\label{eq:Mdagger-block}
\end{align}

Finally, for $w_i=(y_i,z_i,x_i),\,\, i=1,2$, define $\Delta w:=w_1-w_2$
and
$$
\Delta\mathcal F_{sm} := \mathcal F_{sm}(w_1)-\mathcal F_{sm}(w_2) =
\begin{pmatrix}
\Delta G_f\\
\Delta G_g\\
0
\end{pmatrix}.
$$
By the range inclusions established above, we have $\Delta G_f\in\operatorname{Im}(\mathcal P_f)$ and $\Delta G_g\in\operatorname{Im}(\mathcal P_g)$. Therefore, the construction preceding \eqref{eq:Mdagger-var} implies that
$\Delta\mathcal F_{sm}\in\operatorname{Im}(\mathcal M)$. Hence, from \eqref{eq:Mdagger-block}, one has
\begin{equation}\label{eq:Fsm-Mdagger-block}
\|\Delta\mathcal F_{sm}\|_{\mathcal M^\dagger}^2 =\|\Delta  G_f\|_{{\cal P}_f^\dagger}^2 +\|\Delta G_g\|_{{\cal P}_g^\dagger}^2.
\end{equation}
Combining \eqref{eq:f-coco-Pf}, \eqref{eq:g-coco-Pg}, and \eqref{eq:Fsm-Mdagger-block}, we obtain
$$
\begin{aligned}
\left\langle\Delta\mathcal F_{sm},\Delta w
\right\rangle
&= \left\langle \Delta G_f,\Delta y\right\rangle
+\left\langle\Delta G_g,\Delta z\right\rangle \\
&\geq \frac{1}{2\beta}\left(\|\Delta G_f\|_{{\cal P}_f^\dagger}^2
+\|\Delta G_g\|_{{\cal P}_g^\dagger}^2\right)
\\
&=\frac{1}{2\beta} \|\Delta\mathcal F_{sm}\|_{\mathcal M^\dagger}^2.
\end{aligned}
$$
This proves both requirements in Assumption~\ref{assume-coco}.
\end{proof}

We now use Assumption~\ref{assume-coco} to establish the key metric property of the fixed-point mapping $\mathbf T$. Although the preconditioning operator $\mathcal M$ may be singular, the cocoercivity of the smooth component in the $\mathcal M^\dagger$-metric allows the perturbation induced by majorization to be controlled in terms of the fixed-point displacement. This yields the following strengthened nonexpansiveness estimate in the $\mathcal M$-seminorm.

\begin{proposition}\label{prop-nonexpansive}
Suppose Assumptions~\ref{assume-wellposed} and \ref{assume-coco} hold with $\beta\in(0,1]$.  Then, for all $w_1,w_2\in\mathbb H$,
\begin{equation}\label{eq:T-averaged}
\|\mathbf T w_1-\mathbf T w_2\|_{\mathcal M}^2 \leq \|w_1-w_2\|_{\mathcal M}^2 -(1-\beta)
\|(I-\mathbf T)w_1-(I-\mathbf T)w_2\|_{\mathcal M}^2.
\end{equation}
In particular, $\mathbf T$ is nonexpansive in the $\mathcal M$-seminorm, including at the endpoint $\beta=1$.
\end{proposition}

\begin{proof}
Set $w_i^+:=\mathbf T w_i$, $\Delta w:=w_1-w_2$, $\Delta w^+:=w_1^+-w_2^+$, $e:=\Delta w-\Delta w^+$, and
$\Delta\mathcal F_{sm}:=\mathcal F_{sm}(w_1)-\mathcal F_{sm}(w_2)$.
The definition of $\mathbf T$ gives
$$
\mathcal M(w_i-w_i^+)-\mathcal F_{sm}(w_i)
\in\mathcal R_{ns}(w_i^+),
\qquad i=1,2.
$$
The monotonicity of $\mathcal R_{ns}$ and the polarization identity imply
\begin{equation}\label{eq:nonexp-base}
\|\Delta w^+\|_{\mathcal M}^2
\leq \|\Delta w\|_{\mathcal M}^2-
\|e\|_{\mathcal M}^2
-2\langle\Delta\mathcal F_{sm},\Delta w^+\rangle.
\end{equation}
By Assumption~\ref{assume-coco},
$$
-2\langle\Delta\mathcal F_{sm},\Delta w\rangle
\leq-\frac1\beta
\|\Delta\mathcal F_{sm}\|_{\mathcal M^\dagger}^2.
$$
Moreover, \eqref{assume-range} permits the weighted generalized Young
inequality
$$
2\langle\Delta\mathcal F_{sm},e\rangle
\leq \frac1\beta
\|\Delta\mathcal F_{sm}\|_{\mathcal M^\dagger}^2
+\beta\|e\|_{\mathcal M}^2.
$$
Since
$-2\langle\Delta\mathcal F_{sm},\Delta w^+\rangle
=-2\langle\Delta\mathcal F_{sm},\Delta w\rangle
+2\langle\Delta\mathcal F_{sm},e\rangle$, inserting these two bounds into
\eqref{eq:nonexp-base} proves \eqref{eq:T-averaged}.
\end{proof}

The estimate in Proposition~\ref{prop-nonexpansive} immediately yields a Fej\'er-type descent property when one of the reference points is chosen from $\operatorname{Fix}(\mathbf T)=\Omega^*$. As a consequence, the fixed-point iteration enjoys both monotonicity of the $\mathcal M$-distance to the solution set and asymptotic regularity in the $\mathcal M$-seminorm.

\begin{lemma}\label{lem:fejer}
Suppose Assumptions~\ref{assume-wellposed} and \ref{assume-coco} hold with $\beta\in(0,1)$. Let $\{w^k\}$ be generated by
$w^{k+1}=\mathbf T w^k$. Then, for every $w^*\in\Omega^*$ and all $k\geq0$,
\begin{equation}\label{eq:fejer-descent}
\|w^{k+1}-w^*\|_{\mathcal M}^2 \leq
\|w^k-w^*\|_{\mathcal M}^2
-
(1-\beta)\|w^{k+1}-w^k\|_{\mathcal M}^2.
\end{equation}
Consequently, $\{w^k\}$ is $\mathcal M$-Fej\'er monotone with respect to $\Omega^*$ and $\mathcal M$-asymptotically regular, i.e.,
$$
\lim_{k\to\infty}\|w^{k+1}-w^k\|_{\mathcal M}=0.
$$
\end{lemma}

\begin{proof}
By setting $w_1 = w^k$ and $w_2 = w^* \in \Omega^*$ in Proposition~\ref{prop-nonexpansive}, and noting that $\mathbf{T}w^* = w^*$ from Theorem~\ref{th-fixM}, we immediately obtain
$$
\|w^{k+1} - w^*\|_{\mathcal{M}}^2 \leq 
\|w^k - w^*\|_{\mathcal{M}}^2 - (1 - \beta) \|w^k - w^{k+1}\|_{\mathcal{M}}^2. 
$$
Since $\beta \in (0,1)$, we have $\|w^{k+1} - w^*\|_{\mathcal{M}} \leq \|w^k - w^*\|_{\mathcal{M}}$, which establishes the global $\mathcal{M}$-Fejér monotonicity.  Furthermore, summing the inequality over $k = 0$ to $N$ yields
$$
(1 - \beta) \sum_{k=0}^N \|w^k - w^{k+1}\|_{\mathcal{M}}^2 
\leq \|w^0 - w^*\|_{\mathcal{M}}^2 - \|w^{N+1} - w^*\|_{\mathcal{M}}^2 
\leq \|w^0 - w^*\|_{\mathcal{M}}^2. 
$$
Letting $N \to \infty$, the series converges, which directly implies $\lim_{k \to \infty} \|w^{k+1} - w^k\|_{\mathcal{M}} = 0$, proving the $\mathcal{M}$-asymptotic regularity.
\end{proof}

Lemma~\ref{lem:fejer} gives asymptotic regularity only in the degenerate seminorm.  The range condition and the ambient Lipschitz continuity of the resolvent close the missing kernel directions and yield the following finite dimensional strong convergence result.

\begin{theorem}\label{thm:standard-convergence}
Suppose Assumptions~\ref{assume-wellposed} and \ref{assume-coco} hold with $\beta\in(0,1)$.  Then, for every $w^0\in\mathbb H$, the sequence generated by $w^{k+1}=\mathbf T w^k$ converges in the ambient norm to a point $w^*\in\Omega^*$.
\end{theorem}

\begin{proof}
Let $J:=(\mathcal M+\mathcal R_{ns})^{-1}$, and let $L_J$ be a global Lipschitz constant supplied by Proposition~\ref{prop-resolvent}.  For any $w_1,w_2\in\mathbb H$, Assumption~\ref{assume-coco}, the generalized Cauchy--Schwarz inequality, and \eqref{assume-range} give
\begin{equation}\label{eq:F-control-by-M}
\|\mathcal F_{sm}(w_1)-\mathcal F_{sm}(w_2)\|_{\mathcal M^\dagger} \leq 2\beta\|w_1-w_2\|_{\mathcal M}.
\end{equation}
Indeed, this is immediate when the left-hand side vanishes; otherwise one divides the cocoercivity inequality by that quantity.  For vectors in $\operatorname{Im}(\mathcal M)$, spectral calculus also gives
$$
\|v\|\leq\|\mathcal M\|^{1/2}\|v\|_{\mathcal M^\dagger},
$$
whereas $\|\mathcal M d\|\leq\|\mathcal M\|^{1/2}\|d\|_{\mathcal M}$.
Consequently,
\begin{equation}\label{eq:T-ambient-control}
\|\mathbf T w_1-\mathbf T w_2\| \leq C_T\|w_1-w_2\|_{\mathcal M}\,\,\,\text{and}\,\,\,
C_T:=L_J(1+2\beta)\|\mathcal M\|^{1/2}.
\end{equation}

Fix $w^*\in\Omega^*$, which exists by Assumption~\ref{assume-wellposed}. Lemma~\ref{lem:fejer} and \eqref{eq:T-ambient-control} yield, for every $k\geq0$,
$$
\|w^{k+1}-w^*\| =\|\mathbf T w^k-\mathbf T w^*\| \leq C_T\|w^k-w^*\|_{\mathcal M} \leq C_T\|w^0-w^*\|_{\mathcal M}.
$$
Thus $\{w^k\}$ is bounded in the ambient norm.  Moreover, for $k\geq1$,
$$
\|w^{k+1}-w^k\| \leq C_T\|w^k-w^{k-1}\|_{\mathcal M}\longrightarrow0
$$
by Lemma~\ref{lem:fejer}.  Since $\mathbf T$ is continuous, every cluster point $\bar w$ satisfies $\bar w=\mathbf T\bar w$ and hence belongs to $\Omega^*$ by Theorem~\ref{th-fixM}.

It remains to prove uniqueness of the cluster point.  Let $\bar w$ and $\widetilde w$ be two cluster points.  Fej\'er monotonicity with respect to $\bar w$, together with a subsequence converging to $\bar w$, shows that $\|w^k-\bar w\|_{\mathcal M}\to0$.  Passing to a subsequence converging to $\widetilde w$ therefore gives $\mathcal M(\bar w-\widetilde w)=0$.  By \eqref{assume-range}, there is $v\in\mathbb H$ such that
$$
\mathcal F_{sm}(\bar w)-\mathcal F_{sm}(\widetilde w)=\mathcal Mv.
$$
Hence the inner product of this difference with $\bar w-\widetilde w$ is zero.  The cocoercivity inequality then implies
$\mathcal F_{sm}(\bar w)=\mathcal F_{sm}(\widetilde w)$.  The two fixed-point identities now have the same input to the single-valued map $J$:
$$
(\mathcal M-\mathcal F_{sm})\bar w =(\mathcal M-\mathcal F_{sm})\widetilde w,
$$
so $\bar w=\widetilde w$.  A bounded finite-dimensional sequence with a unique cluster point converges to that point, completing the proof.
\end{proof}

\begin{remark}
The strict condition $\beta<1$ is used only for the asymptotic regularity of the standard Picard iteration.  The accelerated Halpern scheme below needs only nonexpansiveness and therefore permits the endpoint $\beta=1$.
\end{remark}

\section{Halpern Acceleration of the Majorized ADMM}\label{sec:HMADMM}

The operator-theoretic reformulation developed in Section~\ref{sec2:operator-theoretic} and the nonexpansiveness property established in Proposition~\ref{prop-nonexpansive} provide a natural basis for applying Halpern-type acceleration to the majorized ADMM. In this section, we incorporate an anchoring step into the fixed-point iteration associated with
$\mathbf T$ and derive a Halpern-accelerated majorized ADMM.
We establish the optimal $\mathcal O(1/k)$ worst-case rate for the fixed-point residual and then transfer this estimate to an $\mathcal O(1/k)$ KKT residual upper bound at the intermediate iterates. The resulting algorithm is presented in
Algorithm~\ref{alg:halpern}.
\begin{algorithm}[ht]
\caption{Halpern-Accelerated Majorized ADMM}
\label{alg:halpern}
\begin{algorithmic}[1]
\REQUIRE Penalty parameter $\sigma>0$, proximal operators $\mathcal S$ and $\mathcal T$ satisfying \eqref{eq:wellposed-curvature}, a constant $\beta\in(0,1]$ satisfying Assumption~\ref{assume-coco}, and $0<\rho\leq2-\beta$. Choose initial anchor point \(w^0 = (y^0, z^0, x^0)\).  Define Halpern parameter \(\lambda_k = \frac{1}{k+2}\). %

\FOR{$k = 0, 1, 2, \dots$}
    \STATE \textbf{Step 1 (Intermediate \(z\)-update):}
    $$
    \begin{aligned}
    \tilde{z}^{k+1}
    = \arg\min_{z} \Bigl\{&q(z)+\hat{g}(z,z^k)
    +\langle B^*x^k,z\rangle\\
    &+\frac{\sigma}{2}\|Ay^k+Bz-c\|^2
    +\frac{1}{2}\|z-z^k\|_{\mathcal T}^2\Bigr\}.
    \end{aligned}
    $$
    
    \STATE \textbf{Step 2 (Intermediate dual \(x\)-update):}
    $$
    \tilde{x}^{k+1} = x^k + \sigma (Ay^k + B\tilde{z}^{k+1} - c).
    $$
    
    \STATE \textbf{Step 3 (Intermediate \(y\)-update):}
    $$
    \begin{aligned}
    \tilde{y}^{k+1}
    = \arg\min_{y} \Bigl\{&p(y)+\hat{f}(y,y^k)
    +\langle A^*\tilde{x}^{k+1},y\rangle\\
    &+\frac{\sigma}{2}\|Ay+B\tilde{z}^{k+1}-c\|^2
    +\frac{1}{2}\|y-y^k\|_{\mathcal S}^2\Bigr\}.
    \end{aligned}
    $$
    
    \STATE \textbf{Step 4 Relaxation step:} 
    $$\widehat w^{k+1} = (1-\rho)w^k+\rho\widetilde w^{k+1}.$$
    
    \STATE \textbf{Step 5 Halpern anchoring step:} 
    $$ w^{k+1} = \lambda_k w^0 + (1-\lambda_k)\widehat w^{k+1}. $$ 
\ENDFOR
\end{algorithmic}
\end{algorithm}

The intermediate updates in Algorithm~\ref{alg:halpern} coincide with one application of the fixed-point mapping $\mathbf T$ defined in \eqref{def-T}. More precisely, by the operator-theoretic reformulation established in Section~2,
$$
\widetilde w^{k+1}=\mathbf T(w^k).
$$
Consequently, the relaxation step can be written as
$$
\widehat w^{k+1}=(1-\rho)w^k+\rho\mathbf T(w^k)=\widehat{\mathbf T}(w^k),
$$
where the relaxed mapping $\widehat{\mathbf T}:\mathbb H\to\mathbb H$ is defined by
\begin{equation}\label{def-hat-T}
\widehat{\mathbf T} := (1-\rho)I+\rho\mathbf T.
\end{equation}
It follows that Algorithm~\ref{alg:halpern} admits the compact fixed-point representation
\begin{equation}\label{def-acc-PPP}
w^{k+1}=\lambda_k w^0+(1-\lambda_k)\widehat{\mathbf T}(w^k),\qquad k=0,1,\ldots.
\end{equation}

For any $\rho>0$, the relaxed mapping $\widehat{\mathbf T}$ has the same fixed-point set as $\mathbf T$. Indeed,
$$
w=\widehat{\mathbf T}(w) \quad\Longleftrightarrow\quad w=(1-\rho)w+\rho\mathbf T(w)
\quad\Longleftrightarrow\quad w=\mathbf T(w).
$$
Hence, by Theorem~\ref{th-fixM}, $\operatorname{Fix}(\widehat{\mathbf T})=\operatorname{Fix}(\mathbf T)=\Omega^*$. Thus, Algorithm~\ref{alg:halpern} is precisely a Halpern-type iteration applied to the relaxed fixed-point mapping $\widehat{\mathbf T}$, with anchor point $w^0$. The convergence-rate analysis of \eqref{def-acc-PPP} requires the relaxed mapping $\widehat{\mathbf T}$ to preserve the nonexpansiveness property of the underlying fixed-point mapping. The strengthened estimate in Proposition~\ref{prop-nonexpansive} yields an explicit admissible range for the relaxation parameter $\rho$.

\begin{proposition}\label{prop-hatT}
Suppose Assumptions~\ref{assume-wellposed} and \ref{assume-coco} hold with $\beta\in(0,1]$. If $0<\rho\leq2-\beta$, then
$\widehat{\mathbf T}$ is nonexpansive with respect to the $\mathcal M$-seminorm. More precisely, for any $w_1,w_2\in\mathbb H$,
\begin{equation}\label{eq:hatT-nonexpansive}
\|\widehat{\mathbf T}w_1-\widehat{\mathbf T}w_2\|_{\mathcal M}^2\leq\|w_1-w_2\|_{\mathcal M}^2
-\frac{2-\beta-\rho}{\rho}\|
(I-\widehat{\mathbf T})w_1-(I-\widehat{\mathbf T})w_2\|_{\mathcal M}^2.
\end{equation}
\end{proposition}

\begin{proof}
Let \(\Delta w := w_1 - w_2\), \(\Delta \mathbf{T} := \mathbf{T}w_1 - \mathbf{T}w_2\). Then
$$
\widehat{\mathbf{T}} w_1 - \widehat{\mathbf{T}} w_2 = (1-\rho)\Delta w + \rho\Delta \mathbf{T},
$$
This, together with $2\langle \Delta w, \Delta \mathbf{T} \rangle_{\mathcal{M}} = \|\Delta w\|_{\mathcal{M}}^2 + \|\Delta \mathbf{T}\|_{\mathcal{M}}^2 - \|\Delta w - \Delta \mathbf{T}\|_{\mathcal{M}}^2$, implies that
$$
\begin{aligned}
\|\widehat{\mathbf{T}} w_1 - \widehat{\mathbf{T}} w_2\|_{\mathcal{M}}^2 & = (1-\rho)\|\Delta w\|_{\mathcal{M}}^2 + \rho\|\Delta \mathbf{T}\|_{\mathcal{M}}^2
- \rho(1-\rho)\|\Delta w - \Delta \mathbf{T}\|_{\mathcal{M}}^2\\
\end{aligned}
$$
Furthermore, by Proposition \ref{prop-nonexpansive} and the fact
$(I-\widehat{\mathbf{T}})w_1 - (I-\widehat{\mathbf{T}})w_2
=\rho (\Delta w-\Delta \mathbf{T})$, we have that
$$
\begin{aligned}
\|\widehat{\mathbf{T}} w_1 - \widehat{\mathbf{T}}w_2\|_{\mathcal{M}}^2
&\leq \|\Delta w\|_{\mathcal{M}}^2
- \rho(2-\rho-\beta)\|\Delta w - \Delta \mathbf{T}\|_{\mathcal{M}}^2 \\
&= \|\Delta w\|_{\mathcal{M}}^2
- \frac{2-\rho-\beta}{\rho}\|(I-\widehat{\mathbf{T}})w_1 - (I-\widehat{\mathbf{T}})w_2\|_{\mathcal{M}}^2.
\end{aligned}
$$
The proof is completed.
\end{proof}

Proposition~\ref{prop-hatT} shows that, for any $0<\rho\le 2-\beta$, the relaxed mapping $\widehat{\mathbf T}$ preserves the nonexpansiveness property required by the Halpern iteration. We now specialize the anchoring parameter to
$$
\lambda_k=\frac{1}{k+2},\qquad k=0,1,\ldots,
$$
and derive the corresponding nonergodic fixed-point residual bound.

\begin{theorem}\label{thm:halpern_rate}
Suppose Assumptions~\ref{assume-wellposed} and \ref{assume-coco} hold with
$\beta\in(0,1]$, and let $0<\rho\leq2-\beta$. If $\{w^k\}$ is generated by
Algorithm~\ref{alg:halpern}, then, for every $w^*\in\Omega^*$ and $k\geq0$,
\begin{equation}\label{eq:halpern-rate}
\|w^k-\widehat{\mathbf T}w^k\|_{\mathcal M}
\leq
\frac{2\|w^0-w^*\|_{\mathcal M}}{k+1}.
\end{equation}
Equivalently,
\begin{equation}\label{eq:halpern-rate-T}
\|w^k-\mathbf T w^k\|_{\mathcal M}
\leq
\frac{2\|w^0-w^*\|_{\mathcal M}}{\rho(k+1)}.
\end{equation}
\end{theorem}

\begin{proof}
Let $\mathcal N:=\ker(\mathcal M)$ and equip the quotient space
$\widehat{\mathbb H}:=\mathbb H/\mathcal N$ with
$$
\langle[w],[v]\rangle_{\widehat{\mathbb H}}
:=\langle w,\mathcal Mv\rangle,
\qquad
\|[w]\|_{\widehat{\mathbb H}}=\|w\|_{\mathcal M}.
$$
This is a finite-dimensional Hilbert space.  Proposition~\ref{prop-hatT}
shows that $w_1-w_2\in\mathcal N$ implies
$\widehat{\mathbf T}w_1-\widehat{\mathbf T}w_2\in\mathcal N$.
Consequently,
$$
\mathfrak T[w]:=[\widehat{\mathbf T}w]
$$
is a well-defined nonexpansive mapping on $\widehat{\mathbb H}$, and
$[w^*]\in\operatorname{Fix}(\mathfrak T)$ for every $w^*\in\Omega^*$.
The quotient classes of Algorithm~\ref{alg:halpern} satisfy the classical
Halpern iteration
$$
[w^{k+1}]=\frac{1}{k+2}[w^0]
+\frac{k+1}{k+2}\mathfrak T[w^k].
$$
The sharp Halpern residual estimate of Lieder
\cite{lieder2021convergence}, applied in $\widehat{\mathbb H}$, gives
$$
\|[w^k]-\mathfrak T[w^k]\|_{\widehat{\mathbb H}}
\leq\frac{2\|[w^0]-[w^*]\|_{\widehat{\mathbb H}}}{k+1},
$$
which is exactly \eqref{eq:halpern-rate}.  Finally,
$I-\widehat{\mathbf T}=\rho(I-\mathbf T)$ gives
\eqref{eq:halpern-rate-T}.
\end{proof}

The optimality claim in Theorem~\ref{thm:halpern_rate} refers to the sharp
worst-case fixed-point residual bound over general nonexpansive mappings.  It
does not assert a matching lower bound for the structured KKT residual below.

\begin{remark}\label{rem:parameter-comparison}
The admissible range of the relaxation parameter $\rho$ reflects the interaction between the majorization-induced perturbation and the degenerate preconditioning structure. Indeed, Proposition~\ref{prop-hatT} shows that
$$
0<\rho\le 2-\beta,
$$
where $\beta\in(0,1]$ is the constant appearing in Assumption~\ref{assume-coco}. Thus, the allowable relaxation range depends explicitly on the cocoercivity of the smooth component relative to the $\mathcal M^\dagger$-metric. In particular, when the combined majorization--proximal terms provide only the curvature required to satisfy Assumption~\ref{assume-coco}, the corresponding value of $\beta$ may be close to one, and the admissible upper bound $2-\beta$ approaches one. By contrast, in the absence of the smooth perturbation, i.e., when $\mathcal F_{sm}\equiv0$, the strengthened estimate in Proposition~\ref{prop-nonexpansive} reduces to the firmly nonexpansive-type bound
$$
\|\mathbf T w_1-\mathbf T w_2\|_{\mathcal M}^2
\le
\|w_1-w_2\|_{\mathcal M}^2
-
\|(I-\mathbf T)w_1-(I-\mathbf T)w_2\|_{\mathcal M}^2,
$$
and the relaxed mapping remains nonexpansive for the full range
$$
0<\rho\le2.
$$
This recovers the parameter regime underlying the reflected choices used in
\cite{lu2024restarted,sun2025accelerating}.

The above distinction is also consistent with the parameter behavior reported for majorized ADMM with indefinite proximal terms. In particular, the studies
\cite{li2016majorized,zhang2020linearly} allow the proximal operators themselves to be indefinite, provided that the combined majorization--proximal structure satisfies suitable curvature conditions. This is closely related to the role of
$\mathcal S+\widehat{\Sigma}_f$ and
$\mathcal T+\widehat{\Sigma}_g$
in Assumption~\ref{assume-coco}.
\end{remark}

The preceding theorem establishes an $\mathcal O(1/k)$ rate for the fixed-point residual in the metric induced by $\mathcal M$. Assumption~\ref{assume-coco} now includes the range compatibility needed to convert this estimate into a bound on the KKT residual in the norm of the ambient space.

\begin{corollary}\label{cor:kkt-rate}
Suppose Assumptions~\ref{assume-wellposed} and \ref{assume-coco} hold with $\beta\in(0,1]$, and let $0<\rho\leq2-\beta$.  Let $\{w^k\}$ be generated by Algorithm~\ref{alg:halpern}, and define the intermediate iterate
$u^k:=\mathbf T(w^k)=\widetilde w^{k+1}$. Then there exists $r^k\in\mathcal R(u^k)$ such that \begin{equation}\label{eq:kkt-residual-bound-general}
\|r^k\| \leq (1+2\beta)\|\mathcal M\|^{1/2} \|w^k-\mathbf T(w^k)\|_{\mathcal M}.
\end{equation}
Moreover, for every $w^*\in\Omega^*$,
\begin{equation}\label{eq:kkt-rate-general}
\operatorname{dist}(0,\mathcal R(u^k)) \leq \frac{2(1+2\beta)\|\mathcal M\|^{1/2}} {\rho(k+1)} \|w^0-w^*\|_{\mathcal M}.
\end{equation}
\end{corollary}

\begin{proof}
Set $d^k:=w^k-u^k. $
By the definition of $\mathbf T$, we have
$$%
0\in \mathcal R_{ns}(u^k) + \mathcal F_{sm}(w^k) + \mathcal M(u^k-w^k).
$$%
Equivalently,
\begin{equation}\label{eq:kkt-rate-proof-rns}
\mathcal M d^k-\mathcal F_{sm}(w^k) \in \mathcal R_{ns}(u^k).
\end{equation}
Adding $\mathcal F_{sm}(u^k)$ to both sides of \eqref{eq:kkt-rate-proof-rns}, we obtain
\begin{equation}\label{eq:kkt-rate-proof-rk}
r^k:=\mathcal M d^k+\mathcal F_{sm}(u^k)-\mathcal F_{sm}(w^k)\in\mathcal R(u^k).
\end{equation}
Hence,
\begin{equation}\label{eq:kkt-rate-proof-dist}
\operatorname{dist}(0,\mathcal R(u^k))\leq \|r^k\|.
\end{equation}

We next estimate the two terms in \eqref{eq:kkt-rate-proof-rk}. First, since $\mathcal M$ is self-adjoint and positive semidefinite,
\begin{equation}\label{eq:Md-bound}
\|\mathcal M d^k\| \leq \|\mathcal M\|^{1/2} \|d^k\|_{\mathcal M}.
\end{equation}
Indeed, this follows from
$
\|\mathcal M d^k\|^2=\langle \mathcal M d^k,\mathcal M d^k\rangle
\leq\|\mathcal M\|\langle d^k,\mathcal M d^k\rangle
=\|\mathcal M\|\|d^k\|_{\mathcal M}^2$.

Second, let $\Delta\mathcal F^k:=\mathcal F_{sm}(u^k)-\mathcal F_{sm}(w^k)$. By Assumption~\ref{assume-coco}, applied to $u^k$ and $w^k$, we have
\begin{equation}\label{eq:coco-uk-wk}
\frac{1}{2\beta}
\|\Delta\mathcal F^k\|_{\mathcal M^\dagger}^{2} \leq \langle \Delta\mathcal F^k,u^k-w^k\rangle .
\end{equation}
Since $\Delta\mathcal F^k\in\operatorname{Im}(\mathcal M)$, the generalized Cauchy--Schwarz inequality associated with the pair $(\mathcal M,\mathcal M^\dagger)$ gives
\begin{equation}\label{eq:generalized-cauchy}
\langle \Delta\mathcal F^k,u^k-w^k\rangle
\leq \|\Delta\mathcal F^k\|_{\mathcal M^\dagger} \|u^k-w^k\|_{\mathcal M}.
\end{equation}
Combining \eqref{eq:coco-uk-wk} and \eqref{eq:generalized-cauchy}, we obtain
\begin{equation}\label{eq:F-Mdagger-bound}
\|\Delta\mathcal F^k\|_{\mathcal M^\dagger} \leq 2\beta \|u^k-w^k\|_{\mathcal M}.
\end{equation}
Moreover, since $\Delta\mathcal F^k\in\operatorname{Im}(\mathcal M)$, the spectral characterization of the Moore--Penrose inverse implies
\begin{equation}\label{eq:ambient-from-Mdagger}
\|\Delta\mathcal F^k\| \leq
\|\mathcal M\|^{1/2} \|\Delta\mathcal F^k\|_{\mathcal M^\dagger}.
\end{equation}
Combining \eqref{eq:F-Mdagger-bound} and \eqref{eq:ambient-from-Mdagger} yields
\begin{equation}\label{eq:F-ambient-bound}
\|\Delta\mathcal F^k\|\leq 2\beta\|\mathcal M\|^{1/2}\|u^k-w^k\|_{\mathcal M}.
\end{equation}

Now, using \eqref{eq:kkt-rate-proof-rk}, \eqref{eq:Md-bound}, and \eqref{eq:F-ambient-bound}, we obtain
$$
\begin{aligned}
\|r^k\|&\leq
\|\mathcal M d^k\|+\|\Delta\mathcal F^k\|\leq(1+2\beta)\|\mathcal M\|^{1/2}\|u^k-w^k\|_{\mathcal M}\\
&=(1+2\beta)\|\mathcal M\|^{1/2}\|w^k-\mathbf T(w^k)\|_{\mathcal M}.
\end{aligned}
$$
This proves \eqref{eq:kkt-residual-bound-general}. Finally, \eqref{eq:kkt-rate-general} follows from \eqref{eq:kkt-rate-proof-dist}, \eqref{eq:kkt-residual-bound-general}, and Theorem~\ref{thm:halpern_rate}.
\end{proof}

\section{Numerical Experiments}\label{sec:numerical}

We first examine the predicted residual decay of Halpern--MajADMM.  We then
study the effects of indefinite proximal terms and majorization, before
comparing a practical implementation with related methods on public convex
quadratic programs.
The methods were implemented in MATLAB R2021b and run in double precision
on an Intel Xeon Gold 6326 CPU, using a single computational thread.
In each comparison, the residual is evaluated at the intermediate output of
the corresponding unrelaxed ADMM map.  The residual definitions and stopping
rules are given in the respective subsections.

\subsection{Validation of the Halpern rate mechanism}
\label{subsec:halpern-rate-mechanism}

We test the $\mathcal O(1/k)$ exact KKT residual rate predicted by
Theorem~\ref{thm:halpern_rate} and Corollary~\ref{cor:kkt-rate} on the
horizon-dependent hard family introduced by Chen et
al.~\cite{chen2026admmkkt}.  Here $K$ denotes the prescribed number of
iterations, while $k=0,\ldots,K-1$ is the iteration index of Section~4
for each fixed problem.  For each $K$, let
$y=(y_s,y_n)$ and $z=(z_s,z_n)$ with $y_s,z_s\in\mathbb R^2$ and
$y_n,z_n\in\mathbb R$, and consider
\begin{equation}\label{eq:hard-family-PK}
 (\mathcal P_K)\qquad
 \min_{y,z}\ 
 \delta_U(y_s)+\delta_{V_K}(z_s)+\phi_K(z_n)
 \quad\mathrm{s.t.}\quad y-z=0,
\end{equation}
where $e_1=(1,0)^\top\in\mathbb R^2$ is the first coordinate unit vector,
\[
 U=\operatorname{span}(e_1),\qquad
 V_K=\operatorname{span}\bigl((\cos\theta_K,\sin\theta_K)\bigr),
 \qquad \theta_K=\frac1{\sqrt K},
\]
and
\[
 \phi_K(t)=\mu_K\sqrt{t^2+\epsilon_K^2},\qquad
 \mu_K=\frac{A_0}{4\sqrt K},\qquad
 \epsilon_K=\frac{\mu_K}{10},\qquad A_0=\frac1{\sqrt2}.
\]
The nearly parallel subspaces create slow fixed-point dynamics, while the
smooth scalar term has gradient Lipschitz constant $10$ and makes the
majorization nontrivial.

We embed \eqref{eq:hard-family-PK} in \eqref{model} with $A=I$, $B=-I$,
$c=0$, and choose the parameters and proximal operators as
\begin{equation}\label{eq:hard-family-endpoint}
 \sigma=\beta=\rho=1,\qquad
 \widehat\Sigma_g=10e_ne_n^\top,\qquad
 \mathcal S=0,\qquad
 \mathcal T=-\tfrac12\widehat\Sigma_g.
\end{equation}
Here $e_n=(0,0,1)^\top\in\mathbb R^3$ is the coordinate unit vector
corresponding to the scalar component $z_n$.  These choices give
$\mathcal P_g=5e_ne_n^\top$ and $\|\mathcal M\|=5$, and satisfy the
well-posedness, range, and cocoercivity conditions of Sections~2--4.
The initial state
\[
 y^0=(0,0,A_0),\qquad z^0=0,\qquad x^0=(A_0,0,0)
\]
satisfies $\|w^0-w^*\|_{\mathcal M}=1$ for $w^*=0$.

Let $\mathbf T_K$ denote the specialization of the fixed-point mapping
$\mathbf T$ in \eqref{def-T} to problem $\mathcal P_K$ under the parameter
choice \eqref{eq:hard-family-endpoint}.  Since $\rho=1$, the relaxed mapping
in Algorithm~\ref{alg:halpern} coincides with $\mathbf T_K$.  Starting from
the same $w^0$, Halpern--MajADMM and its matched MajADMM control use,
respectively,
\begin{equation}\label{eq:hard-family-matched-iterations}
 u^k=\mathbf T_K(w^k),\qquad
 w^{k+1}=\frac1{k+2}w^0+\frac{k+1}{k+2}u^k,\qquad
 w_{\rm c}^{k+1}=\mathbf T_K(w_{\rm c}^k).
\end{equation}
Thus MajADMM differs only by omitting the Halpern anchoring.

\begin{figure}[!t]
\centering
\includegraphics[width=0.70\textwidth]{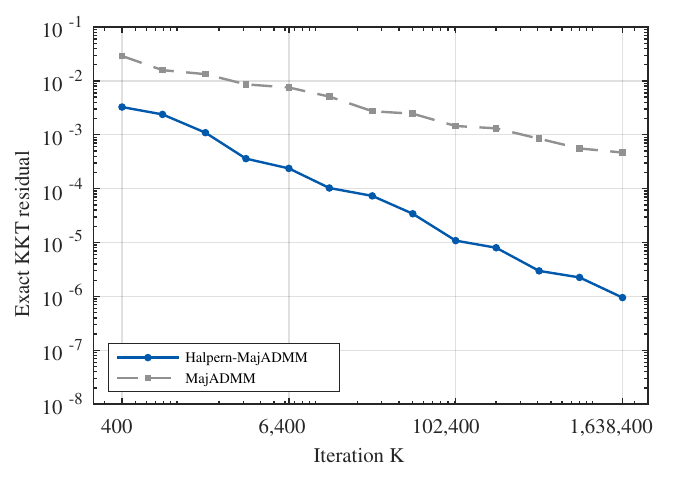}
\caption{Exact KKT residuals on the horizon-dependent family $\mathcal P_K$.}
\label{fig:hard-family-halpern-rate}
\end{figure}

After $K$ iterations, the terminal outputs are $u^{K-1}$ and
$w_{\rm c}^{K}$, at which we evaluate
$R_{\rm KKT}(v)=\operatorname{dist}(0,\mathcal R(v))$.
The bounds in Section~4 at $k=K-1$ become
\begin{equation}\label{eq:hard-family-explicit-bounds}
 \|w^{K-1}-\mathbf T_K(w^{K-1})\|_{\mathcal M}\leq\frac2K,
 \qquad
 \operatorname{dist}\bigl(0,\mathcal R(u^{K-1})\bigr)
 \leq\frac{6\sqrt5}{K}.
\end{equation}

We prespecified the 13 horizons
$K\in\{20^2,\allowbreak 28^2,\allowbreak 40^2,\allowbreak 56^2,\allowbreak
80^2,\allowbreak 112^2,\allowbreak 160^2,\allowbreak 224^2,\allowbreak
320^2,\allowbreak 448^2,\allowbreak 640^2,\allowbreak 896^2,\allowbreak
1280^2\}$.  Each problem $\mathcal P_K$ was run for exactly $K$ iterations.
Each marker in Figure~\ref{fig:hard-family-halpern-rate} therefore represents
a different problem, evaluated at its prescribed horizon.
The fitted tail slope is the least-squares slope of
$\log R_{\rm KKT}$ against $\log K$ over the seven largest horizons; a slope
$s$ corresponds to the power law $R_{\rm KKT}\approx C K^s$.
The anchor remains $w^0$ throughout; no restart or adaptive penalty is used.

Figure~\ref{fig:hard-family-halpern-rate} shows smaller exact KKT residuals
for Halpern--MajADMM at every tested horizon.  The fitted slopes are
$-1.020$ for Halpern--MajADMM and $-0.454$ for MajADMM.  The Halpern iterates
also satisfy both bounds in \eqref{eq:hard-family-explicit-bounds}; their
residual decay is consistent with the predicted $\mathcal O(1/K)$ rate.

\FloatBarrier

\subsection{Effects of indefinite proximal terms and majorization}
\label{subsec:indefinite-majorization}

We first compare the admissible indefinite proximal term with its
semiproximal counterpart, changing only $\mathcal S$ in a fixed-anchor
Halpern--MajADMM iteration.  We then compare majorized and exact nonlinear
subproblem solves under the same pADMM iteration without Halpern anchoring.

For the first comparison, let $m=64$, $y=(y_1,\ldots,y_m)$, and
$z=(z_1,\ldots,z_m)$.  For $j=1,\ldots,m$, set
$U_j=\operatorname{span}(e_1)$ and
$V_j=\operatorname{span}((\cos\theta_j,\sin\theta_j))$, and let $a_j^*$ be a
unit vector.  Set $a^*=(a_1^*,\ldots,a_m^*)$ and define
\[
 \mathcal U_j=a_j^*+U_j,\qquad
 \mathcal V_j=a_j^*+V_j,
 \qquad
 \mathcal U=\prod_{j=1}^m\mathcal U_j,
 \quad
 \mathcal V=\prod_{j=1}^m\mathcal V_j.
\]
With $\Phi(s)=\frac12\sum_{\ell=1}^{2}
\left(s_\ell^2-\log(1+s_\ell^2)\right)$, we solve
\begin{equation}\label{eq:indefinite-test-family}
\begin{aligned}
\min_{y,z}\quad
&\sum_{j=1}^{m}\mu_j\Phi(y_j-a_j^*)
+\delta_{\mathcal U}(y)+\delta_{\mathcal V}(z)\\
\mathrm{s.t.}\quad &y-z=0.
\end{aligned}
\end{equation}
The angles $\theta_j$ are drawn uniformly from
$[0.5^\circ,2.5^\circ]$, and the polar angles of $a_j^*$ are drawn uniformly
from $[0,360^\circ]$.  One quarter of the blocks use $\mu_j=0.05$, and the
remaining blocks use $\mu_j=0.8$.  The solution is $y^*=z^*=a^*$ with zero
multiplier, and $\widehat\Sigma_{f,j}=(9/8)\mu_j I_2$ is a global quadratic
majorant matrix.

For a mapped point $(y,z,x)$, where $x=(x_1,\ldots,x_m)$ is the multiplier,
we evaluate
\begin{equation}\label{eq:indefinite-test-residual}
 R_{\rm KKT}=\left[\frac1m\sum_{j=1}^m\left(
 \begin{aligned}
 &\operatorname{dist}^2(y_j,\mathcal U_j)
 +\operatorname{dist}^2(z_j,\mathcal V_j)+\|y_j-z_j\|^2\\[-1mm]
 &+\|\Pi_{U_j}(\mu_j\nabla\Phi(y_j-a_j^*)+x_j)\|^2
 +\|\Pi_{V_j}x_j\|^2
 \end{aligned}\right)\right]^{1/2}.
\end{equation}

Here $\Pi_C$ denotes the Euclidean projection onto a closed convex set $C$.
In particular, $\Pi_{U_j}$ and $\Pi_{V_j}$ project onto the linear direction
spaces $U_j$ and $V_j$, while the distance terms use the affine sets
$\mathcal U_j$ and $\mathcal V_j$.

We set $\beta=\rho=\sigma=1$, $\mathcal T=0$, and use the fixed-anchor
Halpern iteration in Algorithm~\ref{alg:halpern}.  The common
anchor is the state obtained after 675 MajADMM map evaluations with
$\mathcal S=-\widehat\Sigma_f/2$, starting from the zero state.
The warm-start length was selected on 16
development instances; all settings were then fixed before testing 24
disjoint held-out instances.  The total budget of 3000 iterations includes
the 675 warm-start iterations.  Starting from the common anchor, the paired
runs differ only in
\begin{equation}\label{eq:indefinite-paired-settings}
 \mathcal S=-\tfrac12\widehat\Sigma_f
 \quad\text{versus}\quad
 \mathcal S=0.
\end{equation}
These are the admissible indefinite endpoint and its
semiproximal control.

For the second comparison, let
$\mathcal Y=a+\operatorname{range}(P)$ and
$\mathcal Z=a+\operatorname{range}(Q)$, and consider
the affine-consensus problem
\begin{equation}\label{eq:majorization-cost-family}
\begin{aligned}
\min_{y,z}\quad
&F(y)+G(z)+\delta_{\mathcal Y}(y)+\delta_{\mathcal Z}(z)\\
\mathrm{s.t.}\quad &y-z=0,
\end{aligned}
\end{equation}
where $C,D\in\mathbb R^{r\times n}$ and
\[
 F(y)=\sum_{\ell=1}^{r}\log\cosh((C(y-a))_\ell),\qquad
 G(z)=\sum_{\ell=1}^{r}\log\cosh((D(z-a))_\ell).
\]
The matrices $P,Q\in\mathbb R^{n\times d}$ have orthonormal columns, and all
$d$ principal angles between $\operatorname{range}(P)$ and
$\operatorname{range}(Q)$ equal $\theta$.  The 24 held-out instances use
$n\in\{64,128\}$, $d=n/2$, $\theta\in\{1^\circ,10^\circ,45^\circ\}$,
$r/n\in\{0.25,0.75\}$, and positive-spectrum condition numbers
$\kappa\in\{1,100\}$ for $C^*C$ and $D^*D$.

We compare MajADMM with Exact pADMM, which uses the same update order but
solves the original nonlinear subproblems without majorization.  Since the
second derivative of $\log\cosh$ is bounded by one, $C^*C$ and $D^*D$ are
global majorant matrices, as in the majorization framework
of~\cite{li2016majorized}.  MajADMM uses prefactored linear solves, whereas
Exact pADMM solves the original nonlinear subproblems by safeguarded Newton
iterations to relative first-order residual $10^{-10}$.  Both methods use the
unanchored iteration with $\rho=1$, $\sigma=0.1$, and
$\mathcal S=\mathcal T=0$.  Both start from
$y^0=\Pi_{\mathcal Y}(0)$, $z^0=\Pi_{\mathcal Z}(0)$, and $x^0=0$,
and use the same update order, mapping budget, and KKT evaluator.

At a mapped point $(y,z,x)$, define
\begin{equation}\label{eq:majorization-test-residual}
\begin{aligned}
 \eta_y&=\frac{\|y-\Pi_{\mathcal Y}(y-\nabla F(y)-x)\|}
 {1+\|y\|+\|\nabla F(y)\|+\|x\|},\\
 \eta_z&=\frac{\|z-\Pi_{\mathcal Z}(z-\nabla G(z)+x)\|}
 {1+\|z\|+\|\nabla G(z)\|+\|x\|},\\
 \eta_p &=\frac{\|y-z\|}{1+\|y\|+\|z\|},
 \qquad R_{\rm KKT}=\max\{\eta_y,\eta_z,\eta_p\}.
\end{aligned}
\end{equation}
For each method--instance pair, we report the median elapsed time of three
repetitions.  To compare mean times, we average these medians over the
instances and take the ratio between methods.  Penalized average runtime
with factor two (PAR-2)
assigns twice the setup-inclusive elapsed time after the full 1000-map-call
budget when $R_{\rm KKT}>10^{-5}$ at termination.

Figure~\ref{fig:indefinite-majorization-pairs} presents paired results on
24 held-out instances per problem family.  The left figure compares the
KKT residuals for \eqref{eq:indefinite-test-family} after the total budget
of 3000 iterations: the indefinite choice is on the vertical axis and the
$\mathcal S=0$ control on the horizontal axis.  The right figure compares
the setup-inclusive PAR-2 times for \eqref{eq:majorization-cost-family}
at the KKT target $10^{-5}$, with MajADMM on the vertical axis and
Exact pADMM on the horizontal axis.  In the right figure, circles and squares
denote $n=64$ and $n=128$, respectively.  In both figures, each marker represents one instance,
and points below the diagonal favor the method on the vertical axis.

\begin{figure}[!tbp]
\centering
\includegraphics[width=0.98\textwidth]{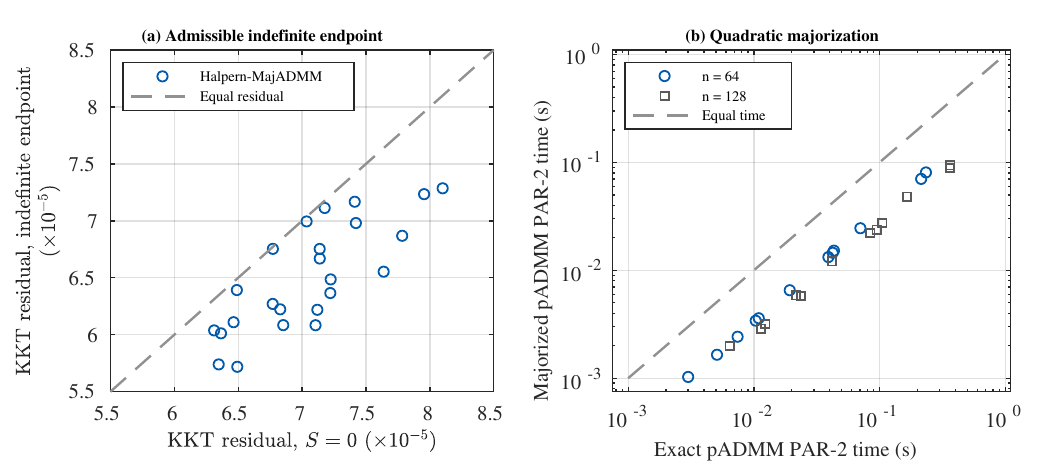}
\caption{Effects of indefinite proximal terms and majorization.}
\label{fig:indefinite-majorization-pairs}
\end{figure}

The indefinite proximal term gives a smaller KKT residual on every instance
in the first comparison.  To summarize residuals across the 24 instances,
we use the root-mean-square (RMS) value
$(24^{-1}\sum_{i=1}^{24}R_i^2)^{1/2}$, where $R_i$ is the residual on
instance $i$.  The RMS residual of the indefinite choice is $0.9229$ times
that of the $\mathcal S=0$ control.

In the second comparison, the two methods have nearly equal RMS residuals
at the three iteration budgets, but MajADMM requires less solver
time per map.  Its RMS residual divided by that of Exact pADMM is $1.0000$
to four decimal places after each of 100, 300, and 1000 iterations.
The mean solver time per map and mean PAR-2 time for MajADMM are
$0.2678$ and $0.2869$ times the corresponding values for Exact pADMM.
MajADMM has a lower PAR-2 time on all 24 instances.  A paired
bootstrap with 10,000 resamples gives a one-sided 95\% upper confidence bound
of $0.3121$ for the PAR-2 ratio.

\FloatBarrier

\subsection{Performance on quadratic programs}
\label{subsec:qp-performance}

We consider convex quadratic programs
\begin{equation}\label{eq:public-qp}
 \min_x\ \frac12\langle x,Qx\rangle+\langle c,x\rangle
 \quad\mathrm{s.t.}\quad Ax=b,\qquad \ell\leq x\leq u,
\end{equation}
where $Q\succeq0$.  The instances are from the public
Maros--M\'{e}sz\'{a}ros collection~\cite{maros1999repository} and are
converted to \eqref{eq:public-qp} by the standard QPPAL preprocessing
of~\cite{liang2022qppal}.  We report QSCTAP2 and QSCTAP3, whose resulting
dimensions $(m,n)$ are $(1090,2500)$ and $(1480,3340)$, respectively.

All methods are evaluated using the same normalized KKT residual for
\eqref{eq:public-qp}.  At $(x,\nu,\mu)$, it is defined by
\begingroup
\medmuskip=2mu
\begin{equation}\label{eq:qp-common-kkt}
 R_{\rm KKT}=\max\biggl\{
 \frac{\|Qx+c+A^*\nu+\mu\|}
 {1+\|c\|+\|Qx\|+\|A^*\nu\|+\|\mu\|},
 \frac{\|Ax-b\|}{1+\|b\|},
 \frac{\|x-\Pi_{[\ell,u]}(x+\mu)\|}
 {1+\|x\|+\|\mu\|}\biggr\},
\end{equation}
\endgroup
where $\nu$ and $\mu$ are the equality and bound multipliers.
The residual is evaluated at the intermediate output of the unrelaxed
ADMM map, before relaxation, Halpern or accelerated averaging, and state
acceptance.  For the majorized splitting, the primal variable in
\eqref{eq:qp-common-kkt} is its mapped variable $y$; for the other
methods, it is their mapped primal variable $x$.
Runs stop when $R_{\rm KKT}\leq10^{-8}$ or after 600 seconds, with no
iteration limit.  Figure~\ref{fig:public-qp-performance} shows the running
minima $\min_{1\leq k\leq K}R_{\rm KKT}(k)$ on QSCTAP2 (left figure)
and QSCTAP3 (right figure), using residuals recorded at every iteration
through $K=3500$.  The horizontal dashed line marks the common tolerance
$10^{-8}$.

\begin{figure}[!tbp]
\centering
\includegraphics[width=0.98\textwidth]{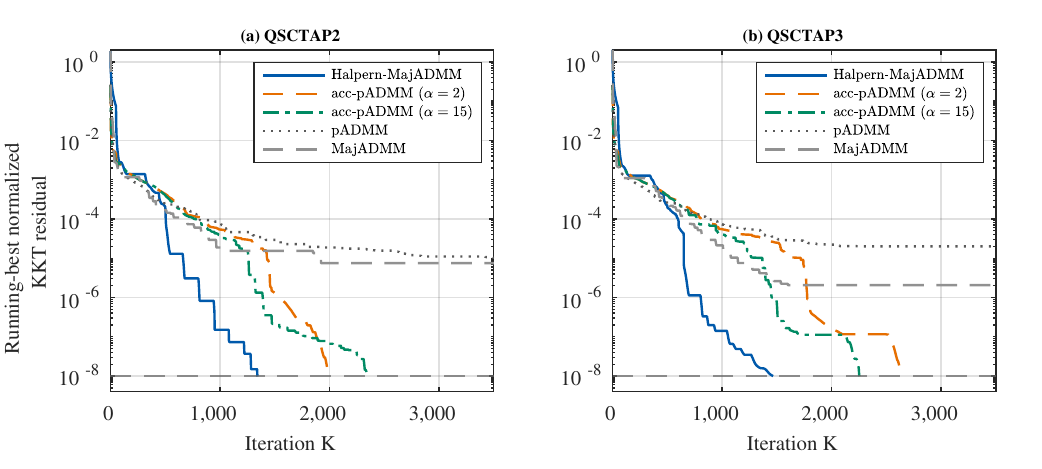}
\caption{KKT residuals on QSCTAP2 and QSCTAP3.}
\label{fig:public-qp-performance}
\end{figure}

Halpern--MajADMM denotes the practical variant with $\beta=\rho=1$ and
$\mathcal S=-Q/2$.  It uses the Halpern update in
Algorithm~\ref{alg:halpern} within each restart cycle.  All variables start
at zero, and $\sigma_0=1$.  Every 50 iterations, the penalty candidate is
$\sigma\sqrt{r}$, where $r$ is $\eta_{\rm p}/\max\{\eta_{\rm d},10^{-12}\}$
clipped to $[10^{-12},10^{12}]$.  Here $\eta_{\rm p}$ is the maximum of the
equality term at $y$ and $\|y-z\|/(1+\|y\|+\|z\|)$, while
$\eta_{\rm d}$ is the maximum of the stationarity term at $y$ and the
box term at $z$, with the normalizations in \eqref{eq:qp-common-kkt}.
The candidate is clipped to $[10^{-6},10^6]$ and accepted only if it differs
from the current penalty by a factor of at least $1.2$.  The multipliers
are unchanged when the penalty changes.

A penalty change restarts Halpern--MajADMM from the accepted iterate.
Otherwise, an HPR-QP-style safeguard~\cite{chen2025hprqp} restarts from the
unrelaxed map output.\footnote{At
the 50-iteration checkpoints, let $d$ be the current fixed-point residual
in the current preconditioner seminorm, and let $d_{\rm ref}$ and
$d_{\rm prev}$ be its first and preceding checked values in the cycle.
The safeguard triggers if $d\leq0.2d_{\rm ref}$, or if
$d\leq0.8d_{\rm ref}$ and $d>d_{\rm prev}$, or if the cycle length reaches
$\max\{100,\lceil\gamma K\rceil\}$.  Here $K$ is the total iteration count;
$\gamma=0.5$, reduced to $0.2$ whenever $d$ is at most $0.1$ times its
initial-run value.  The initial cycle uses the first map residual as its
reference; later cycles initialize the reference at their first checkpoint.}
Each restart updates the anchor and resets the local Halpern index.
The matched MajADMM control uses the same endpoint, map, and penalty rule,
but omits Halpern anchoring and restart.  These practical updates allow the
map to change, unlike the fixed-map experiment in
Section~\ref{subsec:halpern-rate-mechanism}.

We compare with pADMM and accelerated pADMM (acc-pADMM) using the parameter
settings reported by Sun et al.~\cite{sun2025accelerating}.  The pADMM comparison uses
$\rho=1.9$; the two acc-pADMM comparisons use $\rho=2$ and
$\alpha\in\{2,15\}$.  Here $\alpha\geq2$ controls the acceleration weights
in Algorithm~3.1 of~\cite{sun2025accelerating}: $\alpha=2$ gives Halpern
anchoring, whereas $\alpha>2$ gives the fast Krasnosel'ski\u{\i}--Mann
iteration.  This update is applied within each restart cycle.
Penalty updates are checked every 50 iterations using residual balancing
with threshold $0.75$ and multiplicative factors $1.25$ and $0.8$.
The primal measure is the maximum of the equality and box terms in
\eqref{eq:qp-common-kkt} and the quadratic-consistency term
$\|Qx-q\|/(1+\|Qx\|+\|q\|)$, where $q$ is the auxiliary variable
representing $Qx$.  The dual measure is
$\|q+c+A^*\nu+\mu\|/(1+\|c\|)$, as in
\cite{sun2025accelerating}.  The penalty is multiplied by $1.25$ when the
primal measure is below $0.75$ times the dual measure, and by $0.8$ when
the converse holds.  The acc-pADMM variants restart every 200 iterations
and whenever the penalty changes.  All parameter-update and restart
settings were fixed before these tests.
The three baselines are our independent implementations, not reproductions
of the authors' source code.

The penalty and restart rules retain their formulation-specific control
measures; reporting and termination use only the common residual
\eqref{eq:qp-common-kkt}.

Halpern--MajADMM attains $R_{\rm KKT}\leq10^{-8}$ after 1342 iterations on
QSCTAP2 and 1467 iterations on QSCTAP3.  Among the displayed external
baselines, acc-pADMM with $\alpha=2$ reaches the same tolerance in the fewest iterations
on QSCTAP2, after 1997 iterations; on QSCTAP3, acc-pADMM with
$\alpha=15$ does so after 2254 iterations.  On these two instances,
Halpern--MajADMM therefore requires $32.8\%$ and $34.9\%$ fewer iterations,
respectively, than the best-performing external baseline.

\FloatBarrier

\section{Conclusion}
\label{sec:conclusion}

We have developed a Halpern-accelerated majorized ADMM for linearly
constrained convex composite optimization with possibly indefinite proximal
terms.  The perturbed degenerate proximal point representation establishes
nonexpansiveness of the relaxed mapping in the preconditioner seminorm
under the stated assumptions.  Halpern anchoring then gives an
$\mathcal O(1/k)$ fixed-point residual bound and a nonergodic
$\mathcal O(1/k)$ KKT residual bound at the intermediate iterates.
The numerical experiments illustrate the predicted residual decay and the
computational benefits of indefinite proximal terms and quadratic
majorization.  On the two reported public quadratic programs, the practical
variant with adaptive penalties and restarts reaches the common KKT
tolerance in fewer iterations than the compared pADMM and acc-pADMM methods.

\Needspace{6\baselineskip}

\end{document}